\UseRawInputEncoding

\documentclass[11pt,leqno]{article}

\usepackage[all]{xy}
\usepackage{amssymb,amsmath,amsthm,    url,rotating}
\usepackage{mathrsfs}
 
  \usepackage{graphicx,epsfig}
  \usepackage{xcolor,graphicx}
   
    \usepackage{makeidx}
\newcommand{\n}{\noindent}

\newcommand{\vp}{\varepsilon}
\newcommand{\bb}[1]{\mathbb{#1}}
\newcommand{\cl}[1]{\mathcal{#1}}

\newcommand{\ovl}{\overline}

\theoremstyle{plain}
\newtheorem{thm}{Theorem}[section]
\newtheorem{lem}[thm]{Lemma}

\newtheorem{pro}[thm]{Proposition}

\newtheorem{cor}[thm]{Corollary}

\theoremstyle{definition}

\newtheorem{dfn}[thm]{Definition}

\theoremstyle{remark}
\newtheorem{rem}[thm]{Remark}

\numberwithin{equation}{section}

\def\tilde{\widetilde}

\renewcommand{\tilde}{\widetilde}

\def\C{\bb  C}
\def\CC{\bb  C}

\def\E{\bb  E}
\def\F{\bb  F}

\def\d{\delta}

\def\N{\bb  N}

\def\CC{\bb  C}

\def\CC{\bb  C}

\def\E{\bb  E}
\def\F{\bb  F}

\def\d{\delta}

\def\CC{\bb  C}

\def\phi{\varphi}

\def\n{\noindent}
\def\nl{\nolimits}
\def\tr{\rm  tr}

 \def\G{\bb G}

\begin{document}
 
\title{Operator spaces with the  WEP,       the OLLP and the Gurarii property
 }

\author{by\\
Gilles  Pisier  \\
Sorbonne Universit\'e\\
and\\
Texas  A\&M  University}

\def\C{\mathscr{C}}
\def\B{\mathscr{B}}
\def\I{\cl  I}
\def\e{\cl  E}
  
  \maketitle
  
\begin{abstract}    
We  construct  non-exact operator spaces     satisfying     the  Weak  Expectation  Property  (WEP)  and  
 the Operator space version of the Local  Lifting  Property  (OLLP).
 These examples should be compared with the example we recently gave of
 a $C^*$-algebra with WEP and LLP.
The construction produces   several  new analogues among operator spaces
of the Gurarii space, extending Oikhberg's previous work. Each of our
  ``Gurarii operator spaces"  is associated to a class of finite dimensional operator spaces
  (with suitable properties). In each case
 we show the space exists and is unique up to   completely  isometric isomorphism.  
    \end{abstract}

   The main goal of this note is   the construction of  examples of
    operator spaces $X\subset B(H)$ that have both the weak expectation property (in short WEP) and the operator space analogue of the local lifting property (in short OLLP) but are not exact.
    The WEP implies that the bidual  $X^{**}$ is the range of a completely contractive projection
    on $B(H)^{**}$, and hence can be identified with a ternary ring
    of operators (TRO) (see Remark \ref{tro}).
     This construction was known to us before we obtained an example
    of a $C^*$-algebra with both WEP and LLP (see \cite{P7}).
    While the operator space case seems much less significant,
       certain features of the various spaces we obtain, in particular their Gurarii property, may be of independent interest for the theory of triple operator systems  or ternary rings of operators (in short TROs).
       
       The origin of the latter theory
       goes back to Choi and Effros  \cite{[CE3]} (for  unital projections) and  Youngson \cite{Y},
       who proved
       that the range of a completely contractive projection $P$ on a $C^*$-algebra $A$ is
       (completely isometrically) isomorphic to a
       triple subsystem of another $C^*$-algebra $B$ with the triple product
       defined on $P(A)$ by $[a,b,c]=P(ab^*c)$. A ternary ring of operators (TRO in short)
       is a closed  subspace of $B(H,K)$ ($H,K$ being Hilbert spaces)
       stable by the triple product $(a,b,c)\mapsto ab^*c$. 
       A typical example
       of   TRO is the subspace $V=pAq \subset A$  where $p,q$ are (self-adjoint) projections in a $C^*$-algebra $A$. Then $V$ is a left (resp. right) submodule
       with respect to the $C^*$-subalgebras $pAp$ (resp. $qAq$).
   TROs, which were introduced by Hestenes in 1961,
       now have a rather developed theory, we refer to the paper \cite{KR} and the book \cite{BLM} for a detailed exposition including  history and references. 
       
       An operator space $X\subset B(H)$ has the WEP if and only if there is a completely contractive
       projection $P: B(H)^{**} \to X^{**}$ onto $X^{**}$, and  $P$ can always be chosen normal.
       Thus in this case $X^{**}$ is always completely isometric to a TRO by Youngson's theorem.
         We will produce examples of such  spaces with several surprising properties.
         The main one being the ``Gurarii property". In Banach space theory
         the Gurarii space is a separable Banach space $X$
         with the property that for any finite dimensional (f.d. in short)
         Banach space $F$ and any subspace $E\subset F$, for any $\vp>0$
         any linear embedding $f: E\to X$ such that $\|f\|\|{f^{-1}}_{|f(E)}\| <1+\vp$
         admits an extension $\tilde f: F \to X$ still satisfying the strict inequality
         $\|\tilde f\|\|{\tilde f^{-1}}_{|\tilde f(F)} \|<1+\vp$.
         Equivalently, assuming $\|f\|\|{f^{-1}}_{|f(E)}\| \le 1+\vp$, for any $\d>0$
         we can find an extension $\tilde f: F \to X$  satisfying  
                  $\|\tilde f\|\|{\tilde f^{-1}}_{|\tilde f(F)} \|\le (1+\vp)(1+\d)$.
   Gurarii \cite{Gu} showed that such a space exists and is unique up to
   $\vp$-isometric isomorphism, whence the terminology ``Gurarii space". 
   We will denote this space by $\bb G$. Later on Lusky \cite{Lus} showed its uniqueness
   up to isometric isomorphism. 
   More recently in \cite{Kus} Kubi\'s and Solecki (see also \cite{GK})
     gave a much simpler proof
   that greatly influenced \S \ref{uniq} of  the present paper.
   
   We refer to \cite{GK,ACGM} for more  recent results and references
   on generalizations of    Gurarii spaces, which Gurarii originally called
   spaces of almost universal disposition.
   
    In \cite{Oi} Oikhberg
   investigated a non-commutative version of the Gurarii space
   in the operator space framework.
   Operator spaces are defined as closed linear subspaces of $B(H)$,
   between which the completely bounded maps
   are the natural  morphisms,   in contrast with the bounded ones  used  in classical
   Banach space theory (see the books \cite{ER, P4} for details).
   The immediate difficulty that arises in this context is that in sharp contrast with the Banach space case
   the metric space of finite dimensional operator spaces (o.s. in short)
   is not separable for the natural analogue of the Banach-Mazur distance.
   In particular there is no separable o.s. $X$ containing,
   for any $\vp>0$, a completely $(1+\vp)$-isomorphic copy
   of any f.d.o.s. $F$, and this is even impossible  when we restrict to $\dim(F)\le 3$.
   This phenomenon was discovered by Junge and the author in \cite{JP} to disprove  a conjecture of Kirchberg.
    In sharp contract, the characteristic property of $\bb G$  
    obviously implies that it contains a $(1+\vp)$-isomorphic copy of any f.d. Banach space $F$.
     Observing this difficulty,   Oikhberg chose to restrict to the class of $1$-exact o.s.
     that is those o.s. $F$
     admitting for any $\d>0$ a completely $(1+\d)$-isomorphic embedding into
     $B(H)$ for some \emph{finite dimensional} $H$, or equivalently into $M_N$
       for some $N<\infty$. The latter class is clearly separable.
       Oikhberg then proved the analogue of Gurarii's result and Lupini later on
       proved the analogue of Lusky's. Actually Lupini \cite{L0,L1}
       (following Henson and Ben Yaacov \cite{BY})
       developed a much more general theory in which spaces with a Gurarii
       property are described as Fra\"\i ss\'e limits. 
       The examples we present in this paper can be viewed
       as new illustrations of this idea.
       We consider a class of f.d.o.s. $\cl E$
       that is stable  under the operations 
       of passing to the quotient (in the o.s. sense)
       and of  direct sum $(F_1,F_2)\mapsto F_1\oplus_1 F_2$.
       We also assume that any space completely isometric to
       one in $\cl E$ is also in $\cl E$. (Such a class is what we call   a ``league".)
       Lastly, we assume that  $\cl E$ is separable for the
       distance $d_{cb}$. For any such $\cl E$
       we show that there is a \emph{unique} (up to completely isometric isomorphism)
      separable  o.s. 
       $X$ satisfying the cb-analogue of the Gurarii property
       restricted to $F$'s in $\cl E$ and also such that 
       $X$ locally embeds in $\cl E$ in the following sense: for any f.d.
       $E\subset X$  and any $\vp>0$ there is $F\in \cl E$
       containing $E$ completely $(1+\vp)$-isometrically.
       We will denote this unique $X$ by ${\bb G}_\cl E$.
   We show that any separable o.s. $Y$ that   locally embeds in $\cl E$
   embeds completely isometrically in ${\bb G}_\cl E$.
   
   While there are many possibilities,
   our interest focuses on 3 main examples of such $\cl E$'s.\\
   The first one is the class ${\cl E}_{\max}$   of all the f.d. ``maximal" o.s.
   in the sense of \cite{BP} (see also \cite{ER,P4}).
   This is the class of those f.d.  $F\subset B(H)$ such that $\|u\|_{cb}=\|u\|$
   for any $u: F \to Y$ where $Y$ is an arbitrary o.s. This means that 
   $F$ is a f.d. quotient of $\ell_1$ itself equipped with its maximal o.s. structure.
   Equivalently, in place of ${\cl E}_{\max}$ we may consider the class $ \cup_n Q\ell_1^n$
   where $Q\ell_1^n$ denotes the class of all quotients of $\ell_1^n$.
   Indeed, the latter class is suitably dense in ${\cl E}_{\max}$.
            \\    The second example is the class $\cl E_1=\cup_n QS_1^n$
    where $QS_1^n$ denotes the class of all quotients of $S_1^n$ and $S_1^n=M_n^*$
    (o.s. dual). 
    Note that  in sharp contrast with $\ell_1$ for   ${\cl E}_{\max}$ any f.d.o.s. is a quotient of $S_1$
    (see e.g. \cite[p. 27]{BLM}), so
    here we cannot replace $\{S_1^n\mid n\ge 1\}$ by $S_1$.
   \\ The third one is the class $\cl E_{S\C}$ formed of all the f.d. subspaces of the $C^*$-algebra
    $
    \C=C^*(\F_\infty)$ of the free group $\F_\infty$ (with countably infinitely many generators).
   Its stability under $\oplus_1$, under quotients as well
   as under duality
    were observed in \cite{JP}.
    In the latter two cases the space ${\bb G}_\cl E \subset B(H)$ has the WEP,
    whence a  
       projection $P$ from $ B(H)^{**} $ onto $\G^{**}_\cl E$ with $\|P\|_{cb}=1$,
    while for $\cl E =\cl E_{\max}$ we only obtain a $P$ with  $\|P\|=1$.
    In the case $ \cl E=\cl E_1 $, the space ${\bb G}_\cl E$ has the o.s. version of the LLP
    (in short OLLP), which is extensively studied  by Ozawa in \cite{Ozllp}.
    No infinite dimensional $C^*$-algebra has the OLLP.
    Thus ${\bb G}_{{\cl E}_1}$ is an o.s. analogue of the $C^*$-algebra  in \cite{P7}.
    
    The space $\G_{ {\cl E}_{\max} }$  and its bidual  might be relevant to tackle some important open questions
    about maximal operator spaces that we describe in Remark \ref{iq}.
    
    For completeness, we should add to the preceding list   the class ${\cl E}_{\min}$ formed of all f.d.  o.s. that are ``minimal" in the sense of \cite{BP} (see also \cite{ER,P4}).
    This is the class of those f.d.  $F\subset B(H)$ such that $\|u\|_{cb}=\|u\|$
   for any $u: Y \to F$ where $Y$ is an arbitrary o.s. But this example  goes no further than
    the classical case. Indeed, it turns out (see Remark \ref{bex'})
    that the space $\G_{ {\cl E}_{\min} }$ is completely isometric 
    (and a fortiori isometric) to the classical $\G$ equipped with its minimal o.s. structure.
    More interesting variations on the same theme
    are presented in  \S \ref{oik} using Lehner's generalization
    of the minimal/maximal notions \cite{Le}.
     
    We now briefly  describe the contents of this paper.
    In \S \ref{wep} (resp. \S \ref{llp}) we review the WEP (resp. the OLLP).
    In \S \ref{mt} we describe the operator space variant of the so-called ``push out"
    construction. Our main results are in \S \ref{gpos}
    where we produce various examples of o.s. with the Gurarii property,
    in particular a non-exact one with both the WEP and the OLLP.
    In \S \ref{uniq}  we prove the uniqueness (up to complete isometry) of the  Gurarii space  
     associated to any given league. The proof is a simple adaptation of the proof
     given by Kubi\'s and Solecki in \cite{Kus} for the classical Gurarii case.
 In \S \ref{oik} we relate our non-exact examples to Oikhberg's exact Gurarii space.
 
    \medskip
    
    \n{\bf Notation and general background}
    
      \medskip
    
   \n As often, we abbreviate completely bounded by c.b.,
    completely positive by c.p. and completely contractive by c.c.
    We will also abbreviate operator space (or operator spaces) by o.s.
    and finite dimensional by f.d.\\
    An o.s. is a closed subspace of a $C^*$-algebra or of $B(H)$.
    The duality of o.s. is a consequence of Ruan's characterization
    (see e.g.  \cite{ER}) of the sequences of norms on $(M_n(E))_{n\ge 1}$ (where $M_n(E)$ denotes the space of $n \times n$-matrices with entries in a vector space $E$) that come from 
    an embedding of $E$ into $B(H)$ for some $H$.
Given an o.s. $E \subset B(H)$ with (Banach space sense) dual $E^*$  there is an $\cl H$ and an isometric embedding
$j: E^*\to B(\cl H)$  that induces 
isometric isomorphisms $M_n(j(E^*) )\simeq CB(E,M_n)$ for all $n$.
The embedding $j$ allows one to consider $E^*$ as an o.s.
This is what is called the dual o.s. structure on $E^*$.
We will often refer to it as the o.s. dual of $E$.
    We refer to \cite{ER,P4,BLM} for more background on operator spaces.
    Let $(E_i)_{i\in I} $ be a family of o.s. Assuming $E_i\subset B(H_i)$,
    the direct sum in the $\ell_\infty$-sense of $(E_i)_{i\in I} $
    can be realized as embedded ``block diagonally"
    in $B(\oplus_{i\in I} H_i)$. 
    The resulting operator space with be denoted
    by $(\oplus\sum_{i\in I} E_i)_\infty$, and if all the $E_i$'s coincide with a single space $E$ we denote it by $\ell_\infty(I; E)$.
    When $(E_i)_{i\in I} $ is a sequence of spaces $(E_n)$ we use   the lighter notation
    $\ell_\infty(\{E_n\}) $ instead of $(\oplus\sum_{n\in \N} E_n)_\infty$.
    Moreover, we denote by $c_0(\{E_n\}) $ the subspace
    formed of the sequences $x=(x_n)\in \ell_\infty(\{E_n\})$ such that
    $ \lim\|x_n\|=0$.
    We will use the analogous notion of direct sum
    with respect to $\ell_1$.
    The space $(\oplus\sum_{i\in I} E_i)_1$ is the space of families
    $(x_i)_{i\in I} \in \prod_{i\in I} E_i$ with $\sum \|x_i\|_{E_i}<\infty$, equipped with the o.s. structure
    associated to the embedding
    $J: (\oplus\sum_{i\in I} E_i)_1 \subset B(H)$
  defined 
  by
  $$J((x_i)_{i\in I})=  \oplus_{u\in C}\sum\nl_{i\in I} u_i(x_i)$$
  where $C$ is the collection of all $u=(u_i)_{i\in I} \in \prod_{i\in I} \beta_i$
  with 
  $\beta_i$ denoting the unit ball of 
   $  CB(E_i, \cl H)$, and $\cl H$ being 
   a suitably large Hilbert space (say with dimension
   equal to the cardinal of $(\oplus\sum_{i\in I} E_i)_1$).
   Note that
    \begin{equation}\label{pde9}\| J((x_i)_{i\in I})\|=\sup_{\|u_i\|_{cb}\le 1}\|\sum\nl_{i\in I} u_i(x_i)\|.\end{equation}
Therefore
    $\| J((x_i)_{i\in I})\|=\sum\nl_{i\in I} \|x_i\|$
    so that $J$ is an \emph{isometric} embedding of 
    the $\ell_1$-sense (Banach space theoretic) direct sum of $(E_i)_{i\in I}$
    into $B(\cl H)$.
   Similarly for any $(x_i)_{i\in I} \in   M_N((\oplus\sum_{i\in I} E_i)_1)$ we have
    \begin{equation}\label{de9}
    \| (Id_{M_N} \otimes J)((x_i)_{i\in I})\|_{M_N(B(H))}=\sup_{\|u_i\|_{cb}\le 1}\|\sum\nl_{i\in I} (Id_{M_N} \otimes u_i)(x_i)\|_{M_N(B(\cl H))},\end{equation}
    but the latter supremum is less easy to describe than in the case $N=1$.\\
    When $(E_i)_{i\in I} =\{E ,F\}$ we denote 
    $(\oplus\sum_{i\in I} E_i)_1$ simply by $E  \oplus_1 F$.

    When $E_i=\CC$ for any $i\in I$ we set
    $\ell_1(I)=(\oplus\sum_{i\in I} E_i)_1$.
    In that case it is easy to see that
    $\|u\|_{cb}=\|u\|$ for any $u: \ell_1(I) \to B(H)$,
    which means that $\ell_1(I)$ is a maximal o.s.
    
    In particular,   we set
    $$\ell_1^n= \ell_1(\{1,\dots,n\}).$$
    Thus for us $\ell_1^n$ is an o.s.
    
    It is a well known fact that the dual o.s. of
    $(\oplus\sum_{i\in I} E_i)_1$ can be identified
    completely isometrically  with
    $(\oplus\sum_{i\in I} E_i^*)_\infty$.
    In particular we have completely isometrically 
    $${\ell_1^n}^*={\ell_\infty^n}.$$
    
    Let $S^n_1$ (resp. $S_1$) be the trace class on $\ell^n_2$ (resp. $\ell_2$).
      Viewing  $S^n_1$ (resp. $S_1$) as a space formed of $n \times n$-matrices (resp. bi-infinite matrices), we have
      natural embeddings $S^n_1\subset S^{n+1}_1 \subset S_1$. 
    Thus $S_1$ appears as the closure of $\cup_n S_1^n$. Recall $\|x\|_{S_1}={\rm tr}(|x|)$ for all
    $x\in S_1$. Let $E$ be an o.s.
    In \cite{87.b.} a vector valued (meaning here $E$-valued) version of the trace class is considered.
   The space $S^n_1[E]$ is defined as the completion
   of the algebraic tensor product $S^n_1 \otimes E$
   equipped with the  norm defined as follows 
   $$ \forall x\in S^n_1 \otimes E \quad 
   \|x\|_{S^n_1[E]}= \inf\{  \|a\|_{S_2^n} \|y\|_{M_n(E)}\|b\|_{S_2^n} \mid x= (a\otimes I) \cdot y\cdot (b\otimes I)\},$$
   where the ${S_2^n}$-norm    is the Hilbert-Schmidt norm and ${M_n(E)}$ is viewed
   as ${M_n\otimes E}$.
   One can check  that $S^n_1[E] \subset  S^{n+1}_1[E]$ isometrically for any $n$
   and hence we may define $S_1[E]$ simply by setting
   \begin{equation}\label{S1}
   S_1[E]=\ovl{\cup_n S_1^n[E]}.
   \end{equation}
 This space coincides isometrically 
 with the o.s.projective tensor product
 of $S_1$ and $E$    when $S_1$ is equipped with
 its o.s. structure as the dual of the space $K$ of all compact operators on $\ell_2$,
 with the duality defined by $x(k)= \tr ( {}^t x k)$ for $x\in S_1,k\in K$
 (see \cite[\S 4]{P4} \cite[p. 140]{P4} for more details).

    A mapping $u:E \to F$ between o.s. is called completely isometric if it is injective and
    if $\|u\|_{cb}=\|u^{-1}_{|u(E)}\|_{cb}=1$. More generally, let $\vp>0$, we will say that
    $u$ is completely $(1+\vp)$-isometric if $\|u\|_{cb}\le 1+\vp$ and $\|u^{-1}_{|u(E)}\|_{cb}\le1+\vp$.
    In \S \ref{gpos}, for short we  call such maps {\it $\vp$-embeddings}. By {\it an embedding} 
    we mean just a linear embedding, or equivalently an injective map
    between f.d. spaces.
    
    Let $E,F$ be completely isomorphic o.s. We denote by $d_{cb}(E,F)$ the ``multiplicative distance" defined by 
    \begin{equation}\label{dcb} d_{cb}(E,F)=\inf\{\|u\|_{cb} \|u^{-1}\|_{cb}   \}\end{equation}
    where the infimum runs over all isomorphisms $u: E \to F$.
    We set $d_{cb}(E,F)=\infty$ if $E,F$ are not completely isomorphic.
    
    Let $c\ge 1$.
    An operator space $X$ is called $c$-exact if
    for any $\vp>0$ and any f.d. $E \subset X$ there is an integer $N$ and 
    $F\subset M_N$ such that $d_{cb}(E,F)\le c+\vp$.
    Any exact  (in particular any nuclear) $C^*$-algebra is $1$-exact.
    See \cite[p. 285]{P4}
    for more on exact o.s.
    
     We refer the reader to the books \cite{ER,P4}
    for the standard notions of o.s. theory, in particular the notions of dual
    and quotient o.s. that we use freely throughout this paper.

    We end this section with two technical points that are very useful when dealing 
with local (meaning finite dimensional) questions.

\begin{rem}[About perturbation]\label{pertu} We  refer repeatedly to perturbation arguments.
By this we mean the following easily verified assertion
(see \cite[p. 69]{P4}).
Let $\vp>0$. Assume $x_1,\cdots,x_n$ are linearly independent vectors in an o.s.
$X\subset B(H)$. Let    $y_1,\cdots,y_n\in X$
be   such that $\sup_j\|x_j-y_j\|\le \vp$. Then for all small enough $\vp>0$   there is a complete  isomorphism
$w: X \to X$ such that $w(x_j)=y_j$ and such that
$\|w\|_{cb}\|w^{-1}\|_{cb} \le 1+\d(\vp)$
with $\lim_{\vp\to 0}\d(\vp)=0$.
In fact $w$ is a perturbation of the identity on $X$. (Indeed, if $\xi_j\in X^*$ is biorthogonal to $(x_j)$ we may take $w=Id_X+ \sum \xi_j\otimes (y_j-x_j)$.)
 \end{rem}
 \begin{rem}[About inductive limits]\label{ind} 
 Let $(\d_n)$ be a summable sequence of positive numbers.
 Consider a sequence $(E_n)$ of o.s. given together with a sequence
 of embeddings $\phi_n :E_n \to E_{n+1}$ such that
 $\|\phi_n\|_{cb} \le 1+\d_n$ and $\|{\phi^{-1}_n}_{|\phi_n(E_n)}\|_{cb} \le 1+\d_n$.
 The inductive limit $\cl L(\{E_n,\varphi_n\})$ of the system $(E_n ,\phi_n)$ is defined as follows.\\
 Let $\cl L=\ell_\infty(\{E_n\})/c_0(\{E_n\})$ and let $j_n: E_n \to \cl L$ be defined for $x\in E_n$ 
 by $j_n(x)= Q( ( x_k) )$ where $Q: \ell_\infty(\{E_n\})  \to \cl L$ is the quotient map
 and where $(x_k)$ is defined by $x_k=0 $ for all $k<n$, $x_n=x$
 and
 $x_k=\phi_{k-1} \cdots \phi_n(x)$ when $k> n$.
 Note that in the latter case 
 $$((1+\d_{k-1}) \cdots (1+\d_n))^{-1} \|x\| \le \|x_k\| \le (1+\d_{k-1}) \cdots (1+\d_n) \|x\|,$$
 and $\|j_n(x)\|= \limsup_k\|x_k\|$.
  Since we assume $(\d_n)$ summable  $j_n: E_n \to X$ is
  completely $(1+\d_n')$-isometric for some $\d_n'\to 0$
  and
  $d_{cb}(E_n, j_n(E_n))\to 1$. Note $ j_n(E_n) \subset  j_{n+1}(E_{n+1}) \subset \cl L$ for all $n$.
\\
  We then define the inductive limit as $$\cl L(\{E_n,\varphi_n\})=\ovl{\cup j_n(E_n)} \subset \cl L.$$
    The construction shows that the following diagram commutes.
  
  $$\xymatrix{&E_{n+1}\ar[r]^{ j_{n+1} \ \ }& { \ j_{n+1} }(E_{n+1}) \\
& E_n\ \ \ar[u]^{\phi_n} 
 \ar[r]^{j_n} & j_n(E_n)\ar@{^{(}->}[u] }$$
 
 Let $\cl L(\{F_n,\psi_n\})$ be another similar inductive limit associated to
 maps $\psi_n : F_n \to F_{n+1}$.
 Suppose given for each $n$ a map $v_n : F_n \to E_n$ with $\sup\|v_n\|<\infty$ (resp. $\sup\|v_n\|_{cb}<\infty$). 
 The sequence $(v_n)$ defines a bounded (resp. c.b.) ``diagonal" map
 from $\ell_\infty(\{F_n\}) $ to $\ell_\infty(\{E_n\}) $,
 and hence also
 from $\ell_\infty(\{F_n\})/c_0(\{F_n\}) $ to $\ell_\infty(\{E_n\}) /c_0(\{E_n\})$.
 The latter map has norm (resp. c.b. norm) $\le \limsup\|v_n\|$ (resp. $\le \limsup\|v_n\|_{cb}$).
 Moreover if $(v_n)$ is such that
 $v_{n+1} \psi_n= \varphi_n v_n $  for all $n$ (``intertwining"),
 then $(v_n)$ defines a bounded (resp. c.b.)   map
 from $\cl L(\{F_n,\psi_n\}) $ to $\cl L(\{E_n,\varphi_n\}) $.\\
 Lastly, if each $v_n$ is isometric (resp. completely isometric)
 then so is the latter map.

   \end{rem}

       \section{The WEP}\label{wep}
    
  The weak expectation property (in short WEP) for $C^*$-algebras
  was introduced by Lance who gave several equivalent characterizations
  for it. One of them matches 
  the following natural generalization    to operator spaces, which
 was explicitly considered in \cite[p. 270]{BLM}.
  \begin{dfn} An operator space $X\subset B(H)$ has the WEP if there is 
  $T: B(H) \to X^{**}$ with $\|T\|_{cb}=1$ such that $T_{|X}= i_X: X \to X^{**}$, where
    $i_X: X \to X^{**}$ denotes the canonical inclusion.
  \end{dfn}
  Equivalently, the bitransposed embedding $X^{**}\subset B(H)^{**}$ admits a
  projection $P:  B(H)^{**} \to X^{**}$ onto $X^{**}$ such that $\|P\|_{cb}=1$.
  
  This notion does not depend on the (completely  isometric) embedding
  $X\subset B(H)$, as well as its analogue
  for merely contractive projections  considered in the next statement.

   \begin{pro}\label{p1}  
   Let $X\subset B(H)$ be an      o.s. 
   The following are equivalent.
     \item{(i)} There is  a map $T: B(H) \to X^{**}$ with $\|T\|\le 1$
     such that $T(x)=x$ for any $x\in X$.
      \item{(ii)} There is a  contractive projection $P: B(H)^{**} \to X^{**}$ onto $X^{**}$.
      \item{(ii)'} Same as (ii) with   
     a weak* to  weak* continuous projection $P: B(H)^{**} \to X^{**}$.
         \item{(iii)} 
      For  any  $n\ge  1$  and  any  subspace  $S\subset  \ell_1^n$,
          any  linear  map  $u:S  \to  X$  admits
          for  each  $\vp>0$    an  extension
          $\tilde  u:  \ell_1^n\to  X$  with
          $$\|      \tilde  u\|_{cb}  \le  (1+\vp)\|          u\|_{cb}.$$  
              \end{pro}
               \begin{proof}  (i) $\Rightarrow$ (ii)' follows by considering
               $P=({T^*}_{|X^*} )^*$, and  (ii)' $\Rightarrow$ (ii) is trivial.
                Note that (ii) $\Rightarrow$ (i). Assume (i). Let $u$ be as in (iii). Using the injectivity of $B(H)$ we find
                 $\tilde  u:  \ell_1^n\to  X^{**}$ extending $u$ with 
                 $\|      \tilde  u\|_{cb}  =\|          u\|_{cb}.$
               By the local reflexivity principle 
               (which here boils down to the weak* density of
               $B_X$ in $B_{X^{**}}$),  there is a net
               of maps $\tilde  u(i):  \ell_1^n\to  X $ with $\|          \tilde  u(i)\|_{cb}\le 1$
               tending weak* to $\tilde  u$ so that
               $ {\tilde  u (i)}_{|S}$ tends to $u$ for $\sigma(X,X^*)$.
               By Mazur's theorem, passing to  convex combinations
               we get a similar net such that ${\tilde  u(i)}_{|S}$ converges pointwise in norm to $u$. Then by perturbation (see Remark \ref{pertu}), (iii) follows.
                (iii) $\Rightarrow$ (i)  is  a  well  known  application  of      Hahn-Banach. See e.g. \cite[Prop. 2.1]{P7}.
              \end{proof}
              In the o.s. context we have to distinguish between
              contractive  and completely contractive   projections $P: B(H)^{**} \to X^{**}$,
              as follows. \\
              Let $S_1^n=M_n^*$ be the operator space dual of $M_n$, which is  a non-commutative analogue of $ \ell_1^n $.
              
               \begin{pro}\label{p1w}  
   Let $X\subset B(H)$ be an      o.s. 
   The following are equivalent.
     \item{(i)} There is  a map $T: B(H) \to X^{**}$ with $\|T\|_{cb}\le 1$
     such that $T(x)=x$ for any $x\in X$.
      \item{(ii)} There is a  completely contractive projection $P: B(H)^{**} \to X^{**}$ onto $X^{**}$.
      \item{(ii)'} Same as (ii) with   
     a weak* to  weak* continuous projection $P: B(H)^{**} \to X^{**}$.
         \item{(iii)} 
      For  any  $n\ge  1$  and  any  subspace  $S\subset  S_1^n$,
          any  linear  map  $u:S  \to  X$  admits
          for  each  $\vp>0$    an  extension
          $\tilde  u:  S_1^n\to  X$  with
          $$\|      \tilde  u\|_{cb}  \le  (1+\vp)\|          u\|_{cb}.$$  
              \end{pro}
              
              \begin{proof} The proofs that  (i) $\Rightarrow$ (ii)'  (again with
               $P=({T^*}_{|X^*} )^*$),   (ii)' $\Rightarrow$ (ii)  
               $\Rightarrow$ (i) $\Rightarrow$ (iii) are essentially identical to the  ones for the preceding proposition.
                (iii) $\Rightarrow$ (i)   also  follows by    Hahn-Banach
                using the o.s. version of the projective tensor product norm (see e.g. \cite[Prop. 2.1]{P7}). 
                We give the details  for the reader's convenience. 
                Given two o.s. $Y,X$ we denote by
                $Y \otimes_{\wedge} X$ their projective tensor product in  the o.s. sense (see e.g. \cite{ER,P4}).
                Assume (iii).   We claim   that for any o.s. $Y$ 
                the o.s. projective norm
                of the space $ Y^* \otimes B(H)$ induces on its subspace
                 $Y^* \otimes X$ the o.s. projective norm of $Y^* \otimes X$ (i.e.  the norm of $Y^* \otimes_{\wedge} X$). In other words, we claim that
                we have
an {\it isometric} inclusion
$$Y^* \otimes_{\wedge} X  \subset Y^* \otimes_{\wedge} B(H).$$
Let  $t \in Y^* \otimes B(H)$
with $\|t\|_{\wedge}<1$.
We will show that if $t\in Y^* \otimes  X$ then
$\|t\|_{Y^* \otimes_{\wedge} X}<1 $. Since the converse is obvious this will prove our claim that the above inclusion is isometric.\\
 Indeed, $\|t \|_{\wedge}<1$ implies 
(equivalently) that the associated map $v: Y \to B(H)$ admits
for some $n$ a factorization as  $v=ba$ as follows
$$v:  Y {\buildrel a \over \rightarrow S_1^n}
{\buildrel b \over \rightarrow  } B(H)$$ where 
$\|b\|_{cb} \le 1$ and  the tensor $t_a \in Y^* \otimes S_1^n $ associated
to $a$  satisfies
$\|t_a\|_{ Y^* \otimes_{\wedge} S_1^n} <1$. Since $v(Y)\subset X$  
we have $b( a(Y)) \subset X$. Let $S=a(Y)\subset S_1^n$ and let $u=b_{|S}:S \to X$. Let $\tilde  u:  S_1^n\to  X$ be the map given by (iii) (for an $\vp$ to be specified). 
 Clearly $v=\tilde u a$ or equivalently
 $t= ( Id_{Y^*}  \otimes \tilde  u ) ( t_a)$ and hence
  $\|t\|_{Y^* \otimes_{\wedge} X}\le  \|\tilde  u\|_{cb} \|t_a\|_{ Y^* \otimes_{\wedge} S_1^n}  \le (1+\vp) \|t_a\|_{ Y^* \otimes_{\wedge} S_1^n}  $,
  and for $\vp$ chosen small enough this is $<1$, as announced. This proves our claim.
 
When we apply this  to $Y=X$,   the linear form $t\mapsto {\tr}(t)$
is   of unit norm on ${X^* \otimes_{\wedge} X}$, and hence the latter claim implies that
$$\forall t\in X^* \otimes X   \quad |{\tr}(t)| \le \| t\|_{X^* \otimes_{\wedge} B(H)}.$$
Then any Hahn-Banach extension of the linear form
$t\mapsto {\tr}(t)$   of unit norm on ${X^* \otimes_{\wedge} B(H)}$,
defines a complete contraction $T: B(H) \to X^{**}$ such that
$T_{|X} =i_X$. This proves (iii) $\Rightarrow$ (i).   \end{proof}
              
     \begin{cor}\label{cp1}  An o.s.        $X $  has  the  WEP
      if  (and  only  if)  
      for  any  $n\ge  1$  and  any  subspace  $S\subset S_1^n$
          any  linear  map  $u:S  \to  X$  admits
          for  each  $\vp>0$    an  extension
          $\tilde  u:  S_1^n\to  X$  with
          $$\|      \tilde  u\|_{cb}  \le  (1+\vp)\|          u\|_{cb}.$$
              \end{cor}

             \section{The OLLP}\label{llp}
             
  \begin{dfn} 
     We  define  the  LLP  for  a  $C^*$-algebra  $A$
  by  the  equality        $A\otimes_{\min}  {\mathscr{B}}=A\otimes_{\max}  {\mathscr{B}}$,  where
  ${\mathscr{B}}=B(\ell_2)$.  
  \end{dfn}
  Kirchberg    showed  that  this  property  is  equivalent  to  a  
  certain  local  lifting  property  (analogous  to  that  of  $L_1$  in  Banach  space  theory),
  which  has  several  equivalent  forms,  one  of  which  as  follows:
  
  \begin{pro}\label{p3}  A  $C^*$-algebra  $A$  satisfies  $A\otimes_{\min}  {\mathscr{B}}=A\otimes_{\max}  {\mathscr{B}}$
  iff  for  any  $*$-homomorphism  $\pi:  A  \to  C/\I$  into  a  quotient  $C^*$-algebra,
  for  any  f.d.  subspace  $E  \subset  A$  and  any  $\vp>0$  the  restriction  $\pi_{|E}$  admits  a  lifting
  $v:  E  \to  C$  with  $\|v\|_{cb}\le  (1+\vp)$.
  \end{pro}
  \begin{rem} Actually  when  $A$  has  the  LLP  the  preceding  local  lifting  even  holds  with  $\vp=0$. Moreover,
  the conclusion holds for any decomposable $\pi:  A  \to  C/\I$
  with dec-norm equal to $1$ (see \cite[Th. 9.38]{P6}). 
   \end{rem}
  \begin{rem}
 In Kirchberg's original definition in \cite{Kiuhf}, which is equivalent to the above one, 
  a unital $A$ has the LLP if   for any unital c.p. map $\pi: A \to C/\I$
  into  a  quotient  $C^*$-algebra, and any operator system $E  \subset  A$   there is a unital c.p.  map $v:  E  \to  C$ that lifts
  $ \pi_{|E} $.
\\
Similarly,  by his definition $A$  has  the lifting  property  (LP)
    if  any unital c.p. map
      $\pi:  A  \to  C/\I$     admits  a   (``global")  unital  c.p.  lifting.  
  If $A$ is not unital, by definition $A$ has the LLP or the LP if its unitization does.
 Clearly    the LP implies the LLP.
 \\
  Kirchberg \cite{Kiuhf} proved that  the (full) $C^*$-algebra of any countable free group $C^*(\F)$
  has the LP (and a fortiori the LLP). Since any separable $C^*$-algebra $A$ is a quotient of $C^*(\F)$, i.e. we have $A=C^*(\F)/\I$ for some $\I$, 
  the fact that $C^*(\F)$ has the LP (resp. LLP) in Kirchberg's sense
  implies that to check the LP (resp. LLP) for $A$ it suffices to check 
  the liftability (resp. local liftability) of a single bijective $*$-homomorphism
   $\pi: A \to C^*(\F)/\I$. Indeed, the latter implies that the identity of
   $A$ factors (resp. locally factors) via unital c.p. maps through $C^*(\F)$, and hence inherits
   the LP (resp. LLP) from $C^*(\F)$.
    \end{rem}
    
Let $X$ be an operator space, $C$ a $C^*$-algebra, ${\cl I}\subset C$ a (closed,
self-adjoint, 2-sided) ideal in $A$,  with quotient $C^*$-algebra  $C/{\cl I}$.
Consider a completely contractive map  $\varphi\colon X\to C/{\cl I}$. 
We say that $\varphi$ locally lifts (or admits a local lifting)
if for any finite 
dimensional (f.d. in short) subspace $E\subset X$, there exists a completely contractive map 
$\psi\colon E\to C$  such that 
$q\circ\psi =\varphi |_E$, where $q: C \to C/{\cl I}$ is the quotient map.
The OLLP was defined by  Ozawa in \cite{Ozllp}:
\begin{dfn} An o.s. $X$ has the  OLLP if  
any complete contraction $\varphi\colon X\to C/{\cl I}$ into an arbitrary  quotient $C^*$-algebra
locally lifts.
  We say $X$ has the 
 OLP if one can take $E=X$ in the above situation. 
\end{dfn} 
 
Let $C^*\langle X\rangle$ denote the universal $C^*$-algebra of
$X$, characterized by the fact that it contains $X$ as a generating subspace
and any complete contraction on $X$ into $B(H)$ uniquely extends
to a $*$-homomorphism on $C^*\langle X\rangle$ (see e.g. \cite{P4}).
\begin{pro}[\cite{Ozllp}] \label{cor}
The operator space $X$ has OLLP (resp. OLP and is separable) 
if and only if $C^*\langle X\rangle$ has LLP (resp. LP and is separable). 
\end{pro} 

      \begin{rem}
By  \cite[Prop. 2.9]{Ozllp} a separable o.s. $X$ has the OLLP if and only if
any complete contraction into the Calkin algebra
admits a completely contractive lifting.
 \end{rem}
 
  \begin{rem}
The question whether any separable maximal operator space $X$
has the OLP is equivalent to the well known open
problem 
whether any separable quotient $C^*$-algebra $C/{\cl I}$ admits
a contractive lifting $C/{\cl I} \to C$. It is known (see \cite{TBA}) that this holds
if $X$ has the metric approximation property. See \cite{157} for a more recent
account on this topic. A fortiori,
all maximal operator spaces with the metric approximation property 
have the OLLP.
       Whether this holds without the metric approximation property is an important open
       question put forward   by Ozawa,
       who shows  in  \cite[Prop. 3.1]{Ozllp} that this question is equivalent to asking whether any maximal o.s. is locally reflexive
       (see also \cite{[O3]} and \cite{JO}).
\end{rem}
 
  \begin{rem} Examples of o.s. with the OLLP
      include all preduals of von Neumann algebras (the so-called non-commutative
      $L_1$-spaces). When separable these have the OLP. 
        Moreover the Hardy space $H_1$, 
        and the space $H_1(M_*)$ for any von Neumann algebra $M$, 
        have the OLLP (see \cite[Th.4.2(i)]{JLM}).
 \end{rem}
   
   We will use the following characterization of the OLLP
   from \cite{Ozllp}.

   \begin{thm}\label{op1} An operator space $X\subset B(H)$ has the OLLP iff 
   for any $\vp>0$ and any f.d. subspace $E\subset X$ there are  an integer $n$, a quotient $F$ of $S_1^n$ and a factorization $E {\buildrel v\over
\longrightarrow }  F {\buildrel w\over
\longrightarrow }  X$
   of the inclusion $E\subset X$ with 
   $\|v\|_{cb}\|w\|_{cb}<1+\vp$.  
    \end{thm}
    
    \begin{dfn} An o.s. $X$ has the strong OLLP
if   for any $\vp>0$ and any f.d. subspace $E\subset X$ there 
is a   f.d. subspace $E_1$  with $E\subset E_1\subset X$ 
such that there are  an integer $N$ and a quotient $F$ of $S_1^N$  with 
   $d_{cb} (E_1,F)<1+\vp$.  
\end{dfn}

By perturbation arguments (see Remark \ref{pertu}), Theorem \ref{op1} implies:
 
\begin{cor}\label{c1} Assume $X=\ovl{\cup E_n}$ where $(E_n)$ is an increasing sequence of f.d. subspaces, such that all the inclusions $E_n \to X$ satisfy the   factorization in Theorem \ref{op1}.
Then $X$ has the OLLP.
In particular, this holds if, for each $n$, $E_n$ is completely isometric
to a quotient of $S_1^N$ for some $N=N(n)$. In the latter case, $X$ has the strong OLLP.
\end{cor}

We do not know whether the OLLP implies the strong OLLP.

It seems worthwhile to enlarge the preceding definitions like this:

\begin{dfn}\label{D1} A map $u: X \to Y$ between o.s. will be called
OLLP if for any $\vp>0$ and any f.d. subspace $E\subset X$ there are  an integer $n$, a quotient $F$ of $S_1^n$ and a factorization $E {\buildrel v\over
\longrightarrow }  F {\buildrel w\over
\longrightarrow }  Y$
   of the inclusion $u_{|E}$ with 
   $\|v\|_{cb}\|w\|_{cb}<1+\vp$.  (When $u=Id_X$ this means that $X$ has the OLLP.)
\end{dfn}

\begin{rem}
It is easy to see  that if a map $u: X \to Y$ is OLLP then for any complete contraction
$\varphi: Y \to C/\cl I$ (into an arbitrary quotient $C^*$-algebra) the composition $\varphi u: X \to C/\cl I$ is locally liftable.
The converse can also be checked by the same argument as the one used by Ozawa in
\cite{Ozllp} for the case $u=Id_X$.
\end{rem}

Let us denote by $QS_1^n$ the class of o.s. quotients of $S_1^n$.
We denote  $$\cl E_1=\cup_n QS_1^n.$$

\begin{rem}\label{2r1} Let $X\subset B(H)$ be an o.s.   Then the inclusion   $X\subset B(H)$
(or any completely isometric embedding into $B(H)$)
is an OLLP map if and only if  for any f.d. subspace $E\subset X$
there are  an integer $n$ and   $F\in \cup_n QS_1^n$
such that $E$ embeds completely $ (1+\vp)$-isometrically into $F$.
Indeed, this follows from the injectivity of $B(H)$.
\end{rem}
 Anticipating a bit over Definition \ref{rd9},
  we will say that $X$   locally embeds in  $\cl E_1=\cup_n QS^n_1$ if it satisfies
the property in Remark \ref{2r1}.
 
\begin{pro}\label{op2} For an o.s. $X$ that  locally embeds in  $\cl E_1$,
  the WEP implies  the OLLP.
\end{pro}
\begin{proof}  Let $E\subset X$ be a f.d. subspace. Let us  assume $E\subset F$
   (completely  isometrically)  for some  $F\in \cl E_1$.
   If $X$ has the WEP
   the inclusion $v: E\subset X$ extends to $\tilde v: F \to X^{**}$
   with $\|\tilde v\|_{cb} \le 1 $.
   Since $F^*\subset M_n$ for some $n$, (in particular $F^*$ is 1-exact),
   this guarantees that $F^*\otimes_{\min}  X^{**}=(F^*\otimes_{\min}  X)^{**}$ isometrically.
   It follows that there is a net of maps $v_i: F \to X$
   with $\|v_i\|_{cb} \le 1 $ that tends weak* to $\tilde v$. 
   Recall that $\tilde v(E)\subset X$. Passing to convex suitable convex combinations
   we find a similar net such that ${v_i}_{|E}$   tends in norm to  ${\tilde v}_{|E}=v$.
   By norm-perturbation (see Remark \ref{pertu}), we obtain  $w: F \to X$ with $\|w\|_{cb}\le 1+\vp $ such that ${w}_{|E}=u$.
   Whence a factorization $E{\buildrel v\over
\longrightarrow } F {\buildrel w\over
\longrightarrow } X$ as required in Theorem \ref{op1}.
Here we simplified slightly by assuming $v: E \to F$ is completely  isometric,
but it is easy to adapt the argument to cover the completely $ (1+\vp)$-isometric case.
\end{proof}

The preceding statement is analogous to the fact that for 
a $C^*$-algebra that locally embeds in $\C=C^*(\F_\infty)$
(or equivalently locally embeds in the class $\cl E_{S\C}$ appearing in Remark \ref{3/4}) the WEP implies the 
LLP (see \cite[Prop. 3.7]{P7}).

\section{Main tools}\label{mt}

The following key statement   is called the ``push out lemma"
in \cite{Kus}. It seems to have deep roots in category theory, as emphasized in Pedersen's   \cite{Ped}.
In the Banach space context, this same construction 
was first used by Kisliakov \cite{Kis} (without any name for it), and later became the basic building block for the author's construction in \cite{P0} (see also  \cite{BoP}).
It was adapted to the operator space context
by Oikhberg  in \cite{Oi}. Instead of reproducing Oikhberg's argument in \cite{Oi}
 (which uses \eqref{de9})  we choose  a  presentation
 based on the space $S_1[E] $ described  in   \eqref{S1} that (although less self-contained) emphasizes more clearly
the parallelism between the Banach and operator space cases.

\begin{lem}\label{l1} Let $E,L$ be o.s.
 Let $S\subset L$ be a subspace   and let
 $u: S \to E$ be a c.b. map. Let $G_{u}=\{ (s,-us)\mid s\in S\}$.
 Let $E_1=[L \oplus_1 E] / G_{u}$, let $Q: [L \oplus_1 E]   \to [L \oplus_1 E] / G_{u}$
 denote the quotient map and let $j: E \to E_1$ and $\tilde u: L \to E_1$
 be defined by
 $$\forall e\in E\quad j(e)=Q(0\oplus e) \quad\text{   and   }\quad \forall x\in L\quad \tilde u(x)=Q(x\oplus 0).$$
Then $j$ is injective and $\tilde u$ extends $u$ in the sense that $\tilde u_{|S}= j u$ and
$\|\tilde u\|_{cb}\le 1$.  \\
Moreover, if $\|u\|_{cb}\le 1 $   then    $j$ is a complete isometry.\\
Lastly, if $u$ is completely isometric then so is $\tilde u$.\\
More precisely, if  $u$ is injective with $\|u\|_{cb}\le 1$
 then
\begin{equation}\label{e1}   \| {{\tilde {u}}^{-1}\,}_{|\tilde u(L)} \|_{cb}= \| {u^{-1}}_{|u(S)} \|_{cb}    .\end{equation}
\end{lem}

\begin{proof} 
We start by quickly reviewing the Banach space case,
by which we mean the same statement but with
ordinary norms  instead of the cb-ones everywhere.
$$\xymatrix{&L\ar[r]^{ \tilde u }& E_1 \\
& S\ \ \ar@{^{(}->}[u] 
 \ar[r]^{u} & E\ar@{^{(}->}[u]_{j} }$$
The injectivity of $j$ and the assertions $\tilde u_{|S}=u$
and $\|\tilde u\|\le 1$ are immediate. For $e\in E$
we have
$\|je\|=\inf_{s\in S}\{\|s\|+\|e-us\|\}$ and hence (take $s=0$)
$\|je\|\le \|e\|$, but also $\|s\|+\|e-us\|\ge \|s\|+\|e\|-\|us\|$ and hence
$\|je\|\ge \|e\|$ if $\|u\|\le 1$. Thus $j$ is isometric.
We have
$\|\tilde u(x)\|=\inf_{s\in S}\{\|x+s\|+\|us\|\}$  for all $x\in L$.
Assuming $u$ injective, let $c=\| {u^{-1}}_{|u(E)} \| $. Assume $\|u\|\le 1$ so that $c\ge 1$.
Then,
\begin{equation}\label{FF}
\|x+s\|+\|us\|\ge \|x+s\|+c^{-1}\|s\|\ge c^{-1}(\|x+s\|+ \|s\|)\ge c^{-1}  \|x\|.
\end{equation} 
Therefore $\|\tilde u(x)\|\ge c^{-1}  \|x\|$. 
Equivalently   $\| {{\tilde {u}}^{-1}\,}_{|\tilde u(L)} \| \le c$.
Since  $  u$ is the restriction   of $\tilde u$ 
(meaning $\tilde u_{|S}=j u$) we must have $c\le\| {{\tilde {u}}^{-1}\,}_{|\tilde u(L)} \| $.
This gives us  \eqref{e1} for the norms.  \\
To deduce from this Lemma \ref{l1} as stated above
it will be convenient to use the ``vector valued trace class"
defined in \cite{87.b.}, as follows.
Let  $K$ denote the space of compact operators on $\ell_2$.
Let $S_1=K^*$ (the trace class) equipped with its dual o.s. structure.
For any o.s. $X$ we denote $S_1[X]$ the (o.s. sense) projective tensor
product of $S_1$ with $X$, described in \eqref{S1}.
We will use
the following facts.
 Let $u: X \to Y$ be a map between o.s.
Then $\|u\|_{cb}= \|Id \otimes u : S_1[X] \to S_1[Y] \|$. Moreover 
$u$ is completely isometric if and only if 
  $Id \otimes u : S_1[X] \to S_1[Y] $ is isometric.
We have $S_1[X \oplus_1 Y]=S_1[X] \oplus_1 S_1[Y] $ (isometrically)
and for any subspace $S \subset X$ we have
$S_1[X /S]= S_1[X] /S_1[S]$ (isometrically). 
Actually these identities are all completely isometric but we do not need that here.
For all these facts 
we refer to our ast\'erisque memoir \cite{87.b.} or \cite[p. 142]{P4}.

Using this it suffices to apply the first part of the proof  
to $Id \otimes u : S_1[S] \to S_1[E]$ with $ S_1[S] \subset S_1[L]$. 
Observe that by the preceding  facts $[S_1[L] \oplus_1 S_1[E] ]/ S_1[G_{u}]$
can be identified naturally with $S_1[E_1]$. Then
  $\tilde{Id \otimes u}: S_1[L] \to S_1[E_1]$  
 can  be identified with $Id \otimes \tilde{ u}$ and
 the associated isometric embedding  $S_1[E] \to S_1[E_1]$
 can  be identified with $Id \otimes j$.
 With  these identifications, Lemma \ref{l1} becomes immediate by the Banach space case.
\end{proof}

\begin{rem} The space $E_1$  
is the solution of a universal problem.
Given $S\subset L$ and $u: S \to E$, assume that we have an o.s.  space $Z$ and operators
$v: L \to Z$ and $w: E \to Z$ such that 
$v_{|S}=wu$.  
Then there is a \emph{unique} operator $T: E_1\to Z$ such that
 $v=T\tilde u$ and $w=Tj$; moreover,
 we have $\|T\|_{cb} \le \max\{ \|v\|_{cb},\|w\|_{cb}  \}$.
 It is easy to see by diagram chasing that if 
 there exists an $E_1$ satisfying this, then it
  is unique up to complete isometry, and it does exist since the space $E_1$ in
  Lemma \ref{l1} clearly has this universal property.
\end{rem}

\begin{rem}[$E_1/E \simeq L/S  $]
As a supplement to Lemma \ref{l1}: the space  $E_1/E$ is completely isometrically isomorphic to $L/S$. More precisely,
there is a completely isometric isomorphism $\check u: L/S \to E_1/E$ induced by the map
$\tilde u: L \to E_1$. The verification is an easy exercise. The point is that the possible metric lifting  properties of $L\to L/S$ are all inherited by $E_1\to E_1/E$.
\end{rem}
Lemma \ref{l1} admits the following immediate
generalization:

\begin{lem}\label{l1'} Let $E$ be an o.s.
Let $ (L_i)_{i\in I} $ be a family of o.s. with subspaces
   $S_i\subset L_i$  and c.b. maps
 $u_i: S_i \to E$. 
 There is an o.s. $E_1$ and a completely isometric embedding $j : E \subset E_1$
 such that
 for all $i\in I$ there is an operator $\tilde u_i : L_i \to E_1$
 with  $\|{\tilde u_i}\|_{cb}= \|u_i\|_{cb} $  such that that
  $  {\tilde {u}_i\ }_{|S_i}= j u_i$
 for all $i\in I$. Moreover, whenever $u_i$ is injective, so is $\tilde u_i$ and we have
\begin{equation}\label{GG}  \| {{\tilde {u_i}}^{-1}\,}_{|{\tilde u_i}(L_i)} \|_{cb}= \| {u_i^{-1}}_{|u_i(S_i)} \|_{cb}.\end{equation}
In addition, $E_1/E$ is completely isometrically isomorphic to $(\oplus\sum{L_i/S_i})_1$.
 \end{lem}
 \begin{proof}
 Let $u'_i=\|  u_i\|^{-1}_{cb}u_i$.  Let $S=(\oplus\sum_{i\in I} {S_i})_1$, $L=(\oplus\sum_{i\in I}{L_i})_1$
 and let $u: S \to E$ be defined 
 for all $(s_i)  \in S$ by $u((s_i))=\sum  u'_i( s_i)$,
 so that $\|u\|_{cb}=1$. Let  $E_1$, $j:E \to E_1$  (completely isometric) and $\tilde u: L \to E_1$
 be as in Lemma \ref{l1}. 
 Let $v_i: L_i \to E_1$ be the restriction of 
 $\tilde u$ on the $i$-th factor. 
 We have $ \| \tilde u \|_{cb} =1$ and hence  $ \| v_i \|_{cb} \le 1$ for all $i\in I$.
 Let $\tilde u_i= \|  u_i\|_{cb} v_i : L_i \to E_1$ and hence $\| \tilde u_i  \|_{cb} \le \|  u_i\|_{cb}$.
 Note $  {\tilde u_i\ }_{|S_i}= j u_i$ so that $\| \tilde u_i  \|_{cb} = \|  u_i\|_{cb}$.
  To check \eqref{GG} we essentially
 repeat the argument in \eqref{FF}:
 for $x_i\in L_i$ and $s \in S  $  we have 
 $$\|x_i +s_i\|+\sum\nl_{j\not = i} \|s_j\| + \|u'_is_i + \sum\nl_{j\not = i} u'_j s_j \| \ge  \|x_i +s_i\| + \|u'_is_i \| $$ and hence $\|v_ix_i\|\ge  \|{{u'_i}^{-1}}_{|u'_i(S_i)}\|^{-1} \|x_i\|   =
  \|  u_i\|^{-1}_{cb} \|{{u_i}^{-1}}_{|u_i(S_i)}\|^{-1}  \|x_i\| $.
 This implies $ \|{{\tilde u_i\ }^{-1}}_{|\tilde u_i(S_i)}\| \le     \|{{u_i}^{-1}}_{|u_i(S_i)}\|$, and since
 $\tilde u_i$ extends $u_i$ we obtain \eqref{GG} for the norms. Using the spaces
  $S_1[E] $
 as before, we obtain the same with c.b. norms.
 \end{proof}
 
 \begin{rem} One can also deduce Lemma \ref{l1'}
 from Lemma \ref{l1} by iterated applications of
 the latter (as in Remark \ref{iter}) and using transfinite induction after well ordering the set $I$.
  \end{rem}

By homogeneity, Lemma \ref{l1} implies the following:

 \begin{lem}\label{l12/3}  Let $S\subset L$, $E$   be as in     Lemma \ref{l1}.
 Let $v: S\to E$ be an injective c.b. map 
 with $\|  {v^{-1}}_{|  v(S)} \|_{cb} <\infty$ and let $u=v\|v\|^{-1}_{cb}$.
 Let $E_1$ and $j: E \to E_1$ 
  (completely isometric)
 be  associated to $u$ as in Lemma \ref{l1}.
 There is $\tilde v: L \to E_1$ extending $v$ such that
 \begin{equation}\label{12/3}
 \|\tilde v \|_{cb} \|{\tilde v^{-1}}_{|\tilde v(L)} \|_{cb} =  
 \|  v \|_{cb} \|  {v^{-1}}_{|  v(S)} \|_{cb} .\end{equation}
 \end{lem}
 \begin{proof}
 Indeed, we have $\|  \tilde u\|_{cb} =\|  u\|_{cb}=1$, so that if
 we  set  $\tilde v= \tilde u \|v\|_{cb}$ then $\tilde v$ extends $v$ and  
 \eqref{12/3} follows from  \eqref{e1}.
\end{proof}

 \section{The Gurarii property for operator spaces}\label{gpos}

 Let $OS_{fd}$ denote the (complete) metric  space of all f.d. operator spaces
 equipped with the cb-analogue of the Banach-Mazur distance $d_{cb}$
 defined by \eqref{dcb}.
  Two elements $E,F\in OS_{fd}$ are considered the same  if they are completely isometrically isomorphic. This holds if and only if $d_{cb}(E,F)=1$ (or equivalently $\log d_{cb}(E,F)=0$).
  We will often identify an $E\in OS_{fd}$ with a ``concrete" representative $E\subset B(H)$.
  Although slightly abusive these conventions should not lead to any confusion.
  
 Let $\cl E \subset OS_{fd}$ be a set (or ``a class") of finite dimensional o.s.  We will always assume implicitly that   $\cl E $ is non-void and 
 $\cl E\not=\{  \{0\}  \}$. 
 
 A class $\cl E$ will be called separable if it is so with
   respect to the metric $d_{cb}$, i.e. if there is a countable $\cl E'\subset \cl E$
 such that for any $E\in \cl E$ and any $\vp>0$ there is $E'\in \cl E'$ with
 $d_{cb}(E,E')<1+\vp$.
 
For any f.d.o.s. $E\subset B(H)$ we denote
  \begin{equation}\label{e33} d_{cb}(E,\cl E)=\inf \{ d_{cb}(E,F)\mid F \in \cl E\}.\end{equation}
 
   Let $X$ be an o.s. We denote
\begin{equation}\label{e34}d_{SX}(E)=\inf\{ d_{cb}(E,F)\mid F \subset X\}.\end{equation}
 
   \n{\bf Notation:} Let $u :S \to E$ be a map that defines a complete isomorphism 
 into its image. We define the ``distortion" $D(u)$ of $u$  by setting
   \begin{equation}\label{dD}
   D(u)=
   \|  u \|_{cb} \|{ u^{-1}}_{|  u(S)} \|_{cb}.\end{equation}
   Obviously 
   if $v: E \to F$ is a complete isomorphism  from $E$ to $v(E)$
   we have
     \begin{equation}\label{dD'}
   D(vu)\le D(v) D(u) \text{  and also  } D(u)= D( { u^{-1}}_{|  u(S)} ).\end{equation}
Note  $D(t u)=D(u)$ for all $t>0$. Let us denote  
    \begin{equation}\label{21.12}
    D'(u)=
     \max\{ \|  u \|_{cb}  ,  \|{  u^{-1}}_{|  u(S)} \|_{cb}  \},\end{equation}
     so that
\begin{equation}\label{dD''}
  D'(vu)\le D'(v) D'(u) \text{  and also  } D'(u)= D'( { u^{-1}}_{|  u(S)} ).\end{equation}
     Note $D( u) \le    D'(u)^2$. Note also that $D'(u)=D(u)= \|{ u^{-1}}_{|  u(S)} \|_{cb}$ when $\|u\|_{cb}=1$.
 
  \begin{dfn} A map $u$ as above such that 
   $\max\{ \|  u \|_{cb} , \|{ u^{-1}}_{|  u(S)} \|_{cb} \}  \le 1+\vp$
    (i.e. $D'(u) \le 1+\vp$) will be called an $\vp$-embedding.\\
    When $u$ is surjective we say that $u$ is an  $\vp$-isomorphism.
   \end{dfn}
   
 \begin{dfn}\label{d5/2}  Let $\cl E$ be a class of f.d. operator spaces.
 We will say that $\cl E$ is a Gurarii class or simply is Gurarii
 if it satisfies the following property.
 \\
 Consider an injective operator $u: S \to E$
 where $S\subset L$ and  $L,E \in \cl E$.
 Then for any $\d>0$ there is a space $E_1\in \cl E$,
  a completely isometric embedding $\phi: E\to  E_1$ and a map 
  $\tilde u: L \to E_1$ extending $u$ (i.e.   ${\tilde u}_{|S}= \phi u$) such that
 \begin{equation}\label{4/2}
 D(\tilde u )
 \le (1+\d) D(u)
.\end{equation}
 $$\xymatrix{&L\ar[r]^{ \tilde u }& E_1 \\
& S\ \ \ar@{^{(}->}[u] 
 \ar[r]^{u} & E\ar@{^{(}->}[u]_{\phi} }$$
   We will say that $\cl E$ is tightly Gurarii if this holds with $\d=0$.\\
   We will say that $\cl E$ is loosely Gurarii if for any $\vp>0$ this holds for some 
 $\vp$-embedding $\phi: E\to E_1$.
 \end{dfn}
Obviously: tightly Gurarii $\Rightarrow$   Gurarii $\Rightarrow$ loosely Gurarii.
 
  \begin{rem} Lemma \ref{l1} shows that the class of all f.d.  
  operator spaces  is tightly Gurarii, just like the class of all f.d. Banach spaces in the Banach space setting.
 \end{rem}
   \begin{rem}[On iteration]\label{iter}
   Let $\cl E$ be a Gurarii class. Fix $E\in \cl E$.
   The Gurarii property behaves nicely with respect to iteration.
   Indeed, suppose given 
   a finite set of   injective maps $u_i: S_i \to E$ with $S_i\subset L_i$ ($1\le i\le n$).
   We claim that there is $E_n\in \cl E$ and a complete isometry $\Phi_n: E \to E_n$
   for which there are maps ${\tilde u}_i$   such that ${\tilde {u_i}}_{|S_i}$=$\Phi_n u_i$
   for any $1\le i\le n$ and \eqref{4/2} holds for any $u\in \{ u_1,\cdots,u_n\}$.\\
   Indeed, let $E_1$ and $\phi_1: E \to E_1$
     be as in Definition \ref{d5/2} applied to $u=u_1$, and let $\tilde u :L_1 \to E_1$ be the corresponding map. We may reapply the property in Definition \ref{d5/2}
     this time with $E$ replaced by $E_1$ and 
     $u$ replaced by $\phi_1 u_2: S_2 \to E_1$.
   Thus we find $E_2$, a complete isometry  $\phi_2: E_1\to E_2$ 
   and $\tilde{u_2} : L_2 \to E_2$   such that ${\tilde{u_2}}_{|S_2}   =\phi_2\phi_1 u_2$.
   We set $\tilde u_1= \phi_2\tilde u : L_1 \to E_2$. Then
   the map $\Phi_2=\phi_2\phi_1$ satisfies the claim for $n=2$.
   Moreover, \eqref{4/2} holds for any $u\in \{ u_1, u_2\}$.
   Continuing in this way, we obtain complete isometries $\phi_i: E_{i-1}\to E_i$
   so that $\Phi=\phi_n\cdots \phi_1$ satisfies the announced property.
 \end{rem}
   
     \begin{rem}\label{r11/3} We have trivially $\|  u \|_{cb} \le  \|\tilde u \|_{cb}$
  and $\|  {u^{-1}}_{|  u(E)} \|_{cb} \le \|{\tilde u^{-1}}_{|\tilde u(F)} \|_{cb}$
      (since ${\tilde u}_{|S}=u$). Therefore
  \eqref{4/2} implies
 \begin{equation}\label{22.12} 
      \|  \tilde u \|_{cb}\le (1+\d)\|  u \|_{cb} \text{ and }
     \|  {u^{-1}}_{|  u(S)} \|_{cb}\le  (1+\d)\|{\tilde u^{-1}}_{|\tilde u(L)} \|_{cb} .\end{equation}
Conversely, \eqref{22.12} implies \eqref{4/2} with $(1+\d)^2$ in place of $(1+\d)$.
  A fortiori  
        \begin{equation}\label{21.12'} 
     D(\tilde u) \le   (1+\d)D(u) \Rightarrow D'(\tilde u) \le   (1+\d)D'(u).
      \end{equation} 
   \end{rem}
   
  The goal of the next lemma is to show that
  we may apply the Gurarii extension property
  even when the inclusion  
  $S\subset L$ is replaced by  an $\vp$-isometric map $T: Y\to L$.
   
      \begin{lem}\label{18.12} 
     Let $L,E,E_1$ be f.d. operator spaces. Let   $\d>0$.
     Let $\phi: E \to E_1$ be an $\vp$-embedding for some $\vp>0$.
     Assume that for any f.d. subspace $S\subset L$
     and for any injective $u: S \to E$ there is $\tilde u: L \to E_1$
     injective extending $u$ in the sense that $\tilde u_{|S}= \phi u$ and 
     such that $ D'(\tilde u )
 \le (1+\d) D'(u)$. Then for any f.d. Y and 
 any injective map $T: Y \to L $  the following generalization holds:\\
 For any injective $v: Y \to E$ there is $\tilde u: L \to E_1$
     injective extending $v$ in the sense that $\tilde u T= \phi v$ and 
     such that $ D'(\tilde u )
 \le (1+\d) D'(v) D'(T)$. 
  
         \end{lem}
   
    \begin{proof} Consider $S=T(Y)\subset L$ and
    let $u= v {T^{-1}}_{|S}: S \to E$.
    Then by \eqref{dD''} we have $ D'(\tilde u ) \le (1+\d) D'(v {T^{-1}}_{|S})
 \le (1+\d) D'(v) D'(T)$.
 
 \centerline{$\xymatrix{&&L\ar[r]^{ \tilde u }& E_1 \\
&T(Y)\ar@/_1pc/[rr]_{u}\ar@ /^2pc/[ur]\ar[r]^{{T^{-1}}_{|S}} & Y \ar[u]^{T} 
 \ar[r]^{v} & E\ar@{^{(}->}[u]_{\phi} }$}
    \end{proof}
     
 The next definition is a simple modification of a notion considered by Oikhberg in \cite{Oi} to extend the  Gurarii space
 from the Banach space framework to the operator space one.
 
 \begin{dfn}\label{d1} We will say that an operator space $X$ has the $\cl E$-Gurarii property
 if 
 \begin{itemize}\item[{}] for any $\vp>0$ and for any pair of spaces $S\subset L$ with $L\in \cl E$
 the following holds:\\
 for any injective linear map $u: S \to X$
 there is an injective map $\tilde u: L \to X$ extending $u$  such that
 \begin{equation}\label{e32}  D(\tilde u) \le (1+\vp) D(u) .\end{equation}
 \end{itemize} 
 Consequently, if $S\subset L$ with $L\in \cl E$, for any  $\vp,\d>0$ any $\vp$-embedding $u: S \to X$ extends to
 an $(\vp+\d)$-embedding $\tilde u: L \to X$.
 \end{dfn}

  \begin{dfn}\label{rd1} We will say that an o.s. $X$  is   $\cl E$-injective if  
   for any $\vp>0$, any $E\in \cl E$ and  any $S\subset E$, any map $u: S \to X$
   admits an extension $\tilde u: E \to X$ with
   $\|\tilde u\|_{cb}\le (1+\vp)\|u\|_{cb}$.
    \end{dfn}
    
  \begin{rem} 
  For instance when $\cl E$ is the class (denoted below by $\cl E_{S\C}$)
  of all the f.d. subspaces of $C^*(\F_\infty)$, then any WEP o.s. $X$
  is  $\cl E$-injective (see \cite[Remark 21.5]{P6}).
    \end{rem}
 
  \begin{lem}\label{dl5}  
  For an operator space $X$, the $\cl E$-Gurarii property
implies $\cl E$-injectivity.
  \end{lem}
    \begin{proof} 
    By Remark \ref{r11/3} 
    the extension property in Definition \ref{rd1} is satisfied whenever
    the map $u: S \to X$ is injective.
    But since $\dim(S)<\infty$ 
    the set of injective maps is   dense
    in the space $CB(S,X)$.
    By the open mapping theorem (applied to the map $\tilde u\mapsto {\tilde u}_{|S}$)
    this  shows that $X$ is $\cl E$-injective.
      \end{proof}
   \begin{rem}\label{dr1}
Let $\cl {SE}$ be 
   the collection formed of all the 
   o.s. that are (completely isometrically isomorphic to) subspaces of spaces in  $\cl E$.
  It is obvious that  $\cl E$  (resp. loosely) Gurarii  implies  $\cl {SE}$ 
  (resp. loosely) Gurarii.
   Moreover   the $\cl E$-Gurarii property  implies 
   (and hence is equivalent to) the $\cl {SE}$-Gurarii property
   by an immediate restriction argument.  
       \end{rem}
       
 \begin{rem}\label{dr2} Let $\ovl{\cl E}$ be 
   the closure of $\cl E$ for the $d_{cb}$-distance
 (i.e. $E\in\ovl{\cl E}$ if and only if there is a sequence  $(E_n)$ in $\cl E$ such that $d_{cb}(E,E_n)\to 1$). Then the $\ovl{\cl E}$-Gurarii property and the $\cl E$-Gurarii property
  are obviously equivalent.
  By the preceding remark they are also equivalent to the
  $\ovl{\cl {SE}}$-Gurarii property.
  \end{rem}
  
   \begin{rem}\label{dr2'}
   For instance when $\cl E=\{\ell_\infty^n\}$ as in Remark \ref{bex},
   the class   $\ovl{\cl {SE}}$ is equal to the class of all f.d. minimal o.s.
   In the Banach space analogue this is the class of all f.d. spaces.\\
   When $\cl E=\{M_n\}$ the class   $\ovl{\cl {SE}}$ is equal to the class of all $1$-exact  o.s.\\
   When $\cl E=\{R_n\}$  (row matrices) or $\cl E=\{C_n\}$ (column matrices) or when $\cl E$ is the collection of all the f.d. subspaces of a 1-homogeneous 1-Hilbertian o.s. in the sense of \cite[p.172]{P4},
   then $\cl E=\ovl{\cl {SE}}$ (exercise).
  \end{rem}
  
    \begin{dfn} \label{rd9}
We will say that an o.s. $Y$ locally embeds in $\cl E$
if $F\in \ovl{\cl {SE}}$ for
any f.d. $F\subset Y$.
 \\ In other words   for
any f.d. $F\subset Y$  and
 any $\vp>0$
  and  there is $F_1\in \cl E$  and a subspace $\hat F\subset F_1 $ such that $d_{cb} (F,\hat F )\le 1+\vp$.
 \end{dfn}
 
  \begin{rem}\label{dr4}
 Note that, by perturbation, a separable o.s.
 $Y$  locally embeds in $\cl E$
 if and only if  $Y$ is the closure of an increasing union
 of f.d. subspaces $Y_n$ such that $d_{cb}(Y_n, \cl {SE}) \to 1$, and in that case
 actually $d_{cb}(Y_n, \cl {SE}) = 1$ for all $n$.
 In particular, when $Y=\ovl{\cup Y_n}$, if all $Y_n$'s locally embed in $\cl E$
 then so does $Y$.  \end{rem}

 \begin{rem}\label{pdr4}
  Let $Y$ be as in Remark \ref{dr4}. Then   $Y$ is 
  completely isometric to
  an inductive limit 
  of a system $(E_n,\phi_n)$ as in Remark \ref{ind} with $E_n\in \cl {SE}$ for all $n$.
  Indeed, let $(Y_n)$ be as in Remark \ref{dr4}. Let $\d_n >0$ be
  such that $\sum \d_n<\infty$.
  For each $n$ there is $E_n\in \cl {SE}$ and 
  a  $\d_n$-isomorphism   $w_n : Y_n \to E_n$.
   Let  $\phi_n : E_n \to E_{n+1}$ be defined by
  $\phi_n= w_{n+1}w_n^{-1}$, and let $X$ be the inductive limit in the sense
  of Remark \ref{ind}.
  Then the sequence $(w_n)$ defines a completely isometric 
   embedding $Y\to X$.\\
   Conversely, if $d_{\cl {SE}} (E_n)\to 1$ then any inductive limit $X$ as in Remark \ref{ind}
   locally embeds in $\cl E$.
    \end{rem}
 
 \begin{dfn}\label{d2} We will say that an operator space $X$  is an $\cl E$-Gurarii space if it is separable, has the $\cl E$-Gurarii property 
   and locally embeds in $\cl E$.
 \end{dfn}

 \begin{rem}[On the non-separable case]
 Let $E$ be an arbitrary o.s.
 We may apply 
 Lemma \ref{l1'} to the case when $I$ is equal to 
 the collection formed of  \emph{all}  the maps $u: S \to E$ 
 where  $L$ runs over the set of \emph{all} possible f.d. subspaces
 of $B(\ell_2)$ and $S\subset L$ runs over \emph{all} possible subspaces of $L$.
 Let us denote the latter collection by $I(E)$.
 This gives us $E_1\supset E$. Applying the same to $E_1$ and $I(E_1)$, we obtain
 $E_2\supset E_1$ and so on. Let $X=\ovl{\cup E_n}$. It is easy to see that
 $X$ has the $\cl E$-Gurarii property with respect to the class $\cl E$ of \emph{all} f.d.o.s.
A priori this looks like a good non-commutative analogue of the classical Gurarii space.
 Note  however that such a space cannot be separable
  since  $OS_{fd}$ is not separable by \cite{JP}.\\
  Although many arguments below make sense
 in the non-separable case (like the preceding one),
 we restrict to the separable case, which seems to be the main case of interest for us.
  \end{rem}

    \begin{rem} \label{dr3}
   Let $X$  be an $\cl E$-Gurarii space and assume that $Y$ locally embeds in $\cl E$.
  Let $E\subset F\subset Y$ and let $u: E \to X$ be an injective map.
  We observe that for any $\d>0$ there is  
  an injective map $\tilde u: F \to X$ extending $u$  
  such that 
  $D'(\tilde u) \le (1+\d)D'(  u)$  holds.
  Indeed, since   
  $F\in \ovl{\cl {SE}}$ this follows from Remarks \ref{dr1} and \ref{dr2}.
    \end{rem}

   It is rather easy to see that
   any $\cl E$-Gurarii space $X$ is universal-\emph{but only up to $\vp$}-for the class of spaces that locally embed in $\cl E$, meaning that
     if a separable o.s. $Y$   locally embeds in $\cl E$ then
   $Y$ embeds (globally) completely $(1+\vp)$-isometrically into $X$ for any $\vp>0$.
 This is easy to deduce from  Remark \ref{dr3}.
   In Theorem \ref{t4/2p} below we show that this even holds for $\vp=0$ when $\cl E$ is a league. For this we  will use
the   two  lemmas that follow.

  \begin{lem}\label{l11/3}  Let $\cl E$ be a  loosely Gurarii  
   class. Let $E,L_0,\cdots,L_n  \in \cl E$ ($n\ge 0$). For any $\d>0$
   and $\d'>0$  there is $E_1\in \cl E$ and a   $\d'$-embedding $\phi: E\to         E_1$ such that
 for any $L\in \{ L_0,\cdots,L_n\}$, any  $S\subset L$ and any $u: S \to E$
 there is $\tilde u: L \to E_1$ 
 such that ${\tilde u}_{|S}= \varphi u$ and 
  $D(\tilde u) \le (1+\d) D(u)$.\\
  If $\cl E$ is Gurarii we can achieve this with a completely  isometric $\phi: E\to E_1$
  (i.e. with $\d'=0$).
  \end{lem}
    \begin{proof}  Assume first that the set $\{ L_0,\cdots,L_n\}$ is reduced to a single $L\in \cl E$.
The claim in this lemma  is clear for a single $S$ and a single  $u: S \to E$,
 and hence by iteration (as in Remark \ref{iter}) for any finite set of $(S,u)$'s.
A compactness argument will show that this suffices.
 Indeed, by perturbation (see Remark \ref{pertu}), it suffices to show the claim
 for a finite set of subspaces $S\subset L$ generated say by
 points in a $\d''$-net in the unit ball of $L$ for $\d''>0$ chosen suitably small. 
 By homogeneity we may restrict to $u$ such that $\|u\|_{cb}=1$.
 Furthermore, if the claim holds for all $u$ in  a finite $\d''$-net in the unit sphere of $CB(S,E)$
 with $\d''>0$  small enough, we get the same result  
 for all $u: S \to E$.
 We are thus reduced to checking the claim
 for a finite set of $(S,u)$'s for some $\d''$ chosen suitably small. This completes the proof for a single $L$,
 say for $L=L_0$. But actually the discretization 
 we just used can be used identically for each space in the finite set $\{ L_0,\cdots,L_n\}$,
 so the claim follows as stated in the lemma.
 An alternative for the final argument: applying the same result with $E$ replaced by $E_1$ and $L_0$ replaced by $L_1$,
 we obtain the claim in the lemma for $L\in \{ L_0,L_1\}$, and  iterating  further
 we obtain it for any $L\in \{ L_0,\cdots,L_n\}$.
   \end{proof}
    
     \begin{lem}\label{l11/3'}           Let $\cl E, E,  L_0,\cdots,L_n $ be as in Lemma \ref{l11/3}.
     Assume that we are given f.d. spaces $Y_0,Y_1\in \cl{SE}$,
         an      embedding $T_0: Y_0 \to Y_1 $,
     and an embedding  $v_0 : Y_0 \to E $. 
        Then there is $\hat E_1\in \cl E$,
       an embedding $\phi_0: E \to \hat E_1$ with
       $D'(\phi_0) \le (1+\d')^3   D'(T_0)  D'( v_0) $
        and a  $\d'$-embedding $v_1: Y_1\to \hat E_1$
        such that $   v_1T_0= \phi_0 v_0$. In addition,
        for any $L\in \{ L_0,\cdots,L_n\}$, any  $S\subset L$ and any $u: S \to E$
 there is $\hat u: L \to \hat E_1$ 
 such that ${\hat u}_{|S}= \varphi_0 u$ and 
  $D(\hat u) \le (1+\d) D(u) [(1+\d')^2   D'(T_0)  D'( v_0)]^2 $.
       \end{lem}
    \begin{proof}  
    Having obtained $E_1$ as in Lemma \ref{l11/3} we apply 
    the same lemma again  but with  $E$ replaced by $Y_1$ and with the augmented family
    $ \{ L_0,\cdots,L_n\} \cup\{E_1\}$. 
    Moreover, since $\d>0$ is arbitrarily small
    we may take   $\d=\d'$ and $\d'$ as before for this new application. This gives us
    a space $\hat E_1$ and a $\d'$-embedding
    $\psi: Y_1 \to \hat E_1$ such that 
    for any $L\in \{ L_0,\cdots,L_n\} \cup\{E_1\}$, any  $S\subset L$ and any $u: S \to Y_1$
 there is $\tilde u: L \to \hat E_1$ 
 such that ${\tilde u}_{|S}= \psi u$ and $D(\tilde u) \le (1+\d') D(u)$, which by \eqref{21.12'} implies
  $D'(\tilde u) \le (1+\d') D'(u)$. In particular, applying this to $u=T_0$, $S=Y_0$, $L=E_1$
  with the embedding $\phi v_0: Y_0 \to E_1$
   and recalling Lemma \ref{18.12},
  there is $\tilde T_0: E_1 \to \hat E_1$ such that
  $\tilde T_0 \phi v_0= \psi T_0$ and
  $$D'(\tilde T_0) \le (1+\d') D'(T_0) D'(\phi v_0) \le (1+\d') D'(T_0) D'(\phi)D'( v_0) \le  (1+\d')^2   D'(T_0)  D'( v_0).$$
   $$\xymatrix{&&E_1\ar[r]^{ \tilde T_0 }& \hat E_{1}  &   \\
&E\ar[ur]^{\phi}& Y_0\ar[l]^{v_0} \ar[r]^{T_0} \ar[u]_{\phi v_0}
   & Y_1\ar[u]^{\psi}  &  }$$
  We now set $\phi_0= \tilde T_0 \phi: E \to \hat E_1$ and $v_1=\psi$.
  We have $$D'(\phi_0) \le D'(\tilde T_0) D'( \phi) \le (1+\d') D'(\tilde T_0)
  \le (1+\d')^3   D'(T_0)  D'( v_0).$$ Also $   v_1T_0= \phi_0 v_0$ and 
        $D'(v_1)= D'(\psi)\le\d'$. As for the last assertion,
        if $\tilde u$ is as in
        Lemma \ref{l11/3}
       we set $\hat u=\tilde T_0 \tilde u$, so that
       $$D(\hat u) \le D(\tilde u)D(\tilde T_0) \le (1+\d) D(u) D'(\tilde T_0)^2.$$
       This settles the last assertion.
    \end{proof}
 
 The next result shows in particular (say when $Y=\{0\}$) that $\cl E$-Gurarii spaces 
 always exist
   whenever $\cl E$ is  separable and loosely   Gurarii. Our proof is similar to  
   Oikhberg's one   for the ``exact" analogue  \cite[Th. 1.1]{Oi}.
 
  \begin{thm}\label{t4/2p} Let $\cl E$ be a separable loosely   Gurarii class.
  For any  separable o.s.  $Y$ that locally embeds in $\cl E$ there is a
     separable $\cl E$-Gurarii o.s. $X$ containing $Y$ completely isometrically.
 The space $X$ can be written as $X=\ovl{\cup E_n}$
 for some increasing sequence
 $(  E_n)$ of  f.d. subspaces of $X$ such that   $d_{cb}(E_n,\cl E)\to 1$.
 \\ If $\cl E$ is a separable   Gurarii   class, there is an
   $\cl E$-Gurarii o.s.  of the form $X=\ovl{\cup E_n}$ 
    with  $E_n\in \cl E$ for all $n$ such that  for any $\vp>0$ there is an $\vp$-embedding of $Y$ in  $X$. 
 \end{thm}
 
 \begin{proof} By Remark \ref{pdr4} we may view $Y$ as the limit $\cl L(\{Y_n,T_n\})$ of a sequence $Y_n\in \cl{S E}$
 with respect to   $\vp_n$-embeddings  $T_n : Y_n \to Y_{n+1}$,
 with $\sum \vp_n <\infty$. Let $\d'_n$ be decreasing and
  such that $\sum \d'_n<\infty$. Obviously there is a
  positive  sequence $(\d_n)$ such that $\sum \d_n<\infty$
 and
  \begin{equation}\label{17.11}
   ( 1+\d_{n+1}' )^5 (1+\vp_n)^2(1+\d_n')^2 \le 1+\d_{n} \quad  \forall n.\end{equation}
 Starting from $E_0\in \cl E$ chosen arbitrarily,
 we will construct $X$ as a limit of a sequence  $E_n\in \cl{ E}$ 
 with respect to   $\d'_n$-embeddings    $\phi_n : E_n \to E_{n+1}$.
 When $\cl E$ is Gurarii  (and we drop the requirement on $Y$) we can take $\d'_n=0$ so that
  $\phi_n$ is completely  isometric.
   Let $\cl E' =\{L_n\}$ be a dense sequence in $\cl E$.
\\  Assume first that $\cl E$ is Gurarii.
  We will construct $(E_n)$ by induction so that
  for any
  $n\ge 0$
   we have $E_n\in \cl E$ and a complete isometry $\phi_n: E_n \to E_{n+1}$ satisfying  the following condition:
  
 $(*)_n$    For  any  $L\in \{L_0,\cdots,L_n\}$, any
  $S\subset L$ and any $u: S \to E_n$ 
  there is $\tilde u: L \to E_{n+1}$  such that  $\tilde u_{|S}=\phi_n u$
  and $D(\tilde u) \le  (1+\d_n)D(u)$. 
    $$\xymatrix{&L\ar[r]^{ \tilde u }& E_{n+1} \ar[r]^{ j_{n+1} \ \ }& \ j_{n+1}(E_{n+1}) \\
& S\ \ \ar@{^{(}->}[u] 
 \ar[r]^{u} & E_n\ar[u]_{\phi_n} \ar[r]^{ j_{n} } & j_n(E_n) \ar@{^{(}->}[u] }$$
  The verification of the induction step is a consequence of
  Lemma \ref{l11/3}.
  Let then $X=\cl L(\{E_n,\phi_n\})=
   \ovl{\cup j(E_n)}$ as in Remark \ref{ind}.
   Clearly,   $ j_{n}: E_n \to j_n(E_n)$ is completely isometric,
   so that   the inclusion $j_n(E_n) \subset j_{n+1}(E_{n+1})$ has the 
  same property as $\phi_n$ in the preceding diagram. 
 To check the Gurarii property,
 we must consider $L\in \cl E$, $\vp>0$
 and $u:S \to X$ as in \eqref{e32}.
 By the density of $\cl E'$ it suffices to consider $L\in \cl E'$.
 By the density of $ {\cup j_n(E_n)}$ and perturbation we may assume 
 that $u(S)\subset j_m(E_m)$ for some $m$ such that  $\d_m <\vp$.
 A fortiori $u(S)\subset j_n(E_n)$ for all $n\ge m$, and hence
  we may assume that $L\in  \{L_0,\cdots,L_n\}$ and $\d_n <\vp$.
  Then the preceding property of the inclusion $j_n(E_n) \subset j_{n+1}(E_{n+1})$
 gives us the desired Gurarii property. 
 At this point in the proof, the Gurarii property 
 of  the space
 $X $ and an easy induction argument shows that
 for any $\vp>0$ there is an $\vp$-embedding of $Y$ in  $X$. 
 But to produce an $X$ for which this holds for $\vp=0$
 some more care is needed. Curiously, the argument below seems to require
 to modify $X$ (although by Theorem \ref{t5} in many cases
 we have uniqueness). We will need to allow $\varphi_n$'s that are
 $\d_n$-isometric instead of c.i.

 Now if $\cl E$ is only loosely Gurarii, 
 we claim that there is a sequence $(E_n)$ in $\cl E$ and
  $\d_n$-embeddings $\phi_n: E_n \to E_{n+1}$ such that the pair
 $(E_n,\phi_n)$ satisfies $(*)_n$ for all $n\ge 0$,  together with   embeddings
  $v_n : Y_n \to E_n$ such that  $D'(v_n)\le 1+\d'_n $  and
    such that $v_{n+1} T_n= \varphi_n v_n$ for all $n\ge 0$.
Again we set  $X=\cl L(\{E_n,\phi_n\})=
   \ovl{\cup j_n(E_n)}$ as in Remark \ref{ind}.
   Now $ j_{n}: E_n \to j_n(E_n)$ is no longer c.i.
   but we still have $D(j_{n}) \to 1$ so that
the same argument leads to the  Gurarii property for $X$.
Moreover, since $D'(v_n) \to 1$, 
 the sequence $(v_n)$ defines a c.i. embedding
of $Y=\cl L(\{Y_n,T_n\})$ into $X$.\\
To prove our claim, we use induction again.
Since $Y_0\in \cl{SE}$ we find $E_0 \in \cl E$ and    $v_0: Y_0\to E_0$ with $D'(v_0)\le 1+\d'_0$. 
Now let $n\ge 0$, assume that we have found $E_k,\phi_{k-1},v_k$ satisfying the  required bounds and $(*)_{k-1}$ for all $k\le n$.
We must produce $E_{n+1},\phi_{n},v_{n+1}$ satisfying
similar bounds and $(*)_{n}$. In particular
we need 
$D'(\phi_n)\le 1+\d_n$,   $D'(v_{n+1})\le 1+\d_{n+1}'$  and $v_{n+1} T_n =\phi_n v_n$.
When $n=0$ we assume
given only $E_0, v_0$, the argument that follows will produce $E_1, \phi_0, v_1$.
To produce the case $k=n+1$, we first 
invoke Lemma \ref{l11/3'} with $\d,\d'>0$ to be specified :
taking $E=E_n$, this gives us a space            
 a space $\hat E_1\in \cl E$ 
and an  embedding $\phi_{n}: E_n \to \hat E_1$
such that 
for any $L\in \{ L_0,\cdots,L_{n}\}$, any  $S\subset L$ and any $u: S \to E$
 there is $\tilde u: L \to \hat E_1$ 
 such that 
 \begin{equation}\label{17.12}
 {\tilde u}_{|S}= \varphi_n u \text{  and  } 
  D(\tilde u) \le (1+ \d) D(u) [(1+\d')^2   D'(T_n)  D'( v_n)]^2.
\end{equation} 
We also have 
$D'(\phi_n)\le (1+\d')^3   D'(T_n)  D'( v_n) $.
Moreover, we have an embedding $v_{n+1}: Y_{n+1} \to \hat E_1 $ such that 
$v_{n+1}T_n= \phi_n v_n$.
 We have $D'(v_{n+1})\le 1+\d'$.
 
 It remains to check $(*)_{n}$.
From \eqref{17.12}
one finds  the constant in  $(*)_{n}$ majorized by
$$  (1+\d)   [(1+\d')^2   D'(T_n)  D'( v_n)]^2.$$
 To conclude, we choose $\d=\d' =\d_{n+1}'$. This ensures $D'(v_{n+1})\le 1+\d_{n+1}'$.
 Then by \eqref{17.11} we have
 $$  (1+\d)   [(1+\d')^2   D'(T_n)  D'( v_n)]^2 \le ( 1+\d_{n+1}' )^5 (1+\vp_n)^2(1+\d_n')^2  \le 1+\d_{n}$$ 
 which yields the desired constant in $(*)_{n}$.
Moreover,
$$D'(\phi_n)\le (1+\d')^3   D'(T_n)  D'( v_n) \le ( 1+\d_{n+1}' )^3 (1+\vp_n)(1+\d_n') \le 1+\d_{n}.$$
This completes the induction step and the proof.
 \end{proof}

 \begin{rem}[Basic example 1]\label{bex}   Consider the class
   $\cl E=\{\ell_\infty^n \mid n\ge 1\}$. We claim that it is Gurarii.
 To check this let   $u: S \to E$ and $L$ with $L,E\in \cl E$.
 By the Banach space version described at the beginning of  Lemma \ref{l1},
 we know 
 that the class of all f.d.
 spaces is tightly Gurarii, so  there is a f.d. o.s. $D_1$,   an isometric embedding
 $E\subset D_1$ and $\hat u: L \to D_1$ extending $u$ and such that 
 $D(\hat u)     =D(  u) .$
  For any $\d>0$ there is $G\in \cl E$ and an embedding $v: D_1 \to  G$
  such that $\|v\|\le 1$ and $\|{v^{-1}}_{|v(D_1)} \|\le 1+\d$. Since $E$ is injective, there is 
  a projection $P: D_1 \to E$ with $\|P\|=1$. Let  $E_1= G \oplus_\infty E$,
  let $w: D_1 \to E_1$ be defined by $w(x)=v(x) \oplus P(x)$, 
   let $\tilde u= w \hat u$ and $j= w_{|E}$. 
  Then $j$ is isometric  and ${\tilde u}_{|S}=w   u= ju$.
  Moreover, we have $\|{\tilde u^{-1}}_{|\tilde u(L)} \| \le  \|v^{-1}\| \|{\hat u^{-1}}_{|\hat u(L)} \|
  \le 
  (1+\d)  \|{\hat u^{-1}}_{|\hat u(L)} \|$.
  Thus, we have    $$ D(\tilde u )  \le (1+\d) 
  D(\hat u) =
  (1+\d)D(   u) .$$
  Clearly $E_1= G \oplus_\infty E\in \cl E$. 
  This gives us 
  \eqref{4/2} with ordinary norms in place of the c.b. ones.
  But if we equip $L,E_1$ and $E$ with their minimal o.s. structure
  we obtain  \eqref{4/2} with c.b. norms and  $j$ is completely isometric, so that indeed
  $\cl E$  is
  Gurarii.
  \end{rem}
  \begin{rem}\label{bex'}
  When $\cl E=\{\ell_\infty^n \mid n\ge 1\}$  the    space $X$ given by Theorem \ref{t4/2p} is clearly a minimal o.s. (i.e. an operator subspace of a commutative $C^*$-algebra)
  and the underlying Banach space
  possesses  the characteristic property of    the classical Gurarii space $\bb G$. By the uniqueness of the latter (that we will reprove in \S \ref{uniq})
$X\simeq \bb G$ isometrically and hence $X \simeq MIN(\bb G)$ completely isometrically 
where (following the notation in \cite{BP})  $MIN(\bb G)$ denotes  $\bb G$ equipped with its minimal o.s. structure.   Note that by Theorem \ref{t4/2p} 
we have $ \bb G=\ovl {\cup E_n}$ for subspaces $\cdots E_n\subset E_{n+1}\subset \cdots$ such that $E_n\simeq \ell_\infty^{k(n)}$ isometrically. The latter fact is a well known property
of $ \bb G$, see e.g. \cite{LL} for more on this theme.
    \end{rem}
   \begin{rem}[Basic example 2]\label{bex2} Similarly the class $\cl E=\{M_n \mid n\ge 1\}$ is Gurarii. 
  Since $M_n$ is injective the same argument as in Remark \ref{bex} works.
   The space $X$ given by Theorem \ref{t4/2p} is then  completely isometric to Oikhberg's exact Gurarii space (but our argument to show that $ \{M_n \mid n\ge 1\}$ is Gurarii seems simpler). See \S \ref{oik} for more details.
   \end{rem}
  \begin{rem}[An easy exercise]  Let $\cl E=\{R_n\}$ (resp. $\cl E=\{C_n\}$) or let $\cl E$ be the collection of all the f.d. subspaces of a 1-homogeneous 1-Hilbertian o.s.
   $\cl H$  in the sense of \cite[p.172]{P4}. Then $\cl E$ is tightly Gurarii
   and  the space given by Theorem \ref{t4/2p}
is $X=R$ (resp. $X=C$)  or $X=\cl H$.  \end{rem}

    In practise,
    the notion of Gurarii class seems too  ``abstract".  It will be convenient to 
    work with  classes $\cl E$ satisfying certain basic stability properties,
    that we call ``leagues":
    
      \begin{dfn}\label{rass} 
      A class   $\cl E$ will be called a league if it is stable 
      under $\oplus_1$-direct sums and quotients.  Assuming $\{0\}\in \cl E$, this means that
       for any  $E_1,E_2\in \cl E$ 
        and 
       for any subspace $N\subset E_1\oplus_1 E_2$  we have $(E_1\oplus_1 E_2)/N
       \in \cl E$.
     (We also recall   that by convention
     any space that is completely isometric to one in $\cl E$ is also in $\cl E$.)
     \end{dfn}
     
      \begin{pro}\label{p12/3} If $\cl E$  is a league
  then $\cl E$ is tightly Gurarii.  \end{pro}
  \begin{proof} This follows from Lemma \ref{l12/3}.
  \end{proof} 
     
  Our next observation   follows well known ideas going back to the
  Banach version of the Gurarii space:
    by a result due to Lusky \cite{Lus} (improving on Gurarii's initial
    $(1+\vp)$-isometrical uniqueness theorem) the latter space
  is \emph{unique} up to isometric  isomorphism. Our next statement says that
  the $\cl E$-Gurarii space-if it exists-is unique up to   completely   isometric isomorphism.
  In  \cite{Oi}, 
  Oikhberg proved the analogue of Gurarii's uniqueness result in his framework
  involving exact operator spaces, and in  \cite{L0} Lupini 
  proved the analogue of Lusky's result. 
  Imitating the argument of Kubi\'s and Solecki from \cite{Kus}
  we will prove in the next section the following uniqueness:

  \begin{thm}\label{t5}  Assume that $\cl E$ is a separable league.
  Let $X,Y$ be $\cl E$-Gurarii spaces. Then $  X \simeq Y$  
  completely isometrically. 
    \end{thm}
  \begin{proof} This follows from  Corollary \ref{cd1} and Lemma \ref{1}.
     \end{proof}

\begin{rem}[Notation ${\bb G}_{\cl E}$]
By  Theorem \ref{t4/2p} with any $Y\in \cl E$ (or  even with $Y=\{0\}$)   there is an  $\cl E$-Gurarii o.s.
and it is unique by Theorem \ref{t5}.
We will denote it by ${\bb G}_{\cl E}$.
\end{rem}

  \begin{rem}[Some examples of leagues]\label{3/4} Let us denote by $Q\ell_1^n $ (resp. $QS_1^n$)  the class of o.s. quotients of $\ell_1^n$ (resp. $S_1^n$). 
  By Proposition \ref{p12/3}, the following 3 natural classes satisfy the assumptions
  of Theorem \ref{t4/2p}:
  \item{(i)}  $\cl E_{\max}=\cup_{n\ge 1} Q\ell_1^n  $.
   \item{(ii)}  $\cl E_1=\cup_{n\ge 1} QS_1^n$.   
      \item{(iii)}  $\cl E_{S\C}=\{E \ \text{f.d.o.s.} \mid  d_{S\C}(E)=1 \}$, where $d_{S\C}$ is as in \eqref{e34} and where
      $$\C=C^*(\F_\infty).$$
      Note that any f.d. $E$ with $d_{S\C}(E)=1$ embeds completely isometrically in $\C$ (see \cite[p. 352]{P4} or  \cite[Remark 20.8]{P6}), so
      $\cl E_{S\C}$ is simply the class of f.d. subspaces of $\C$.
      \\
      That (i) and (ii) are leagues is obvious.
      \\
        In cases (i) and (ii) the space ${\bb G}_{\cl E}$ has the strong OLLP by Corollary \ref{c1}.
        \end{rem}
      For (iii) the stability under quotient and dual is proved in \cite{JP}. It is an immediate consequence
      of the following lemma. Indeed, by duality, since the dual of $E\oplus_1 F$
      is $  E^*\oplus_\infty F^*$, \eqref{jp} implies:
$$ \forall E,F \ f.d.o.s.  \quad d_{S\C}(E \oplus_1 F) = \max \{d_{S\C}(E), d_{S\C}(F)\},$$
and also for any $N\subset E$ (since $(E/N)^*\subset E^*$):
$$d_{S\C}(E/N)\le d_{S\C}(E).$$

     \begin{lem}[\cite{JP}]
      Let $E\subset B(\ell_2)$ be a f.d.o.s. and let $E^*\subset B(\ell_2)$ be any completely isometric embedding of the dual o.s.
  Then \begin{equation}\label{jp}  d_{S\C}(E)=\sup \{\|t\|_{B(\ell_2)\otimes_{\max}B(\ell_2) } \mid t\in E^* \otimes B(\ell_2), \ \|t\|_{B(\ell_2)\otimes_{\min}B(\ell_2) }=1\}.\end{equation}
        \end{lem}

      \begin{rem} 
      By Theorem \ref{t4/2p}, $\ell_1$ (resp. $S_1$) embeds completely isometrically
      in ${\bb G}_{\cl E}$ when $\cl E=\cl E_{\max}$  (resp. when $\cl E=\cl E_1$).
      Since $\C$ (trivially) locally embeds in $\cl E_{S\C}$, it embeds completely isometrically
   in ${\bb G}_{\cl E}$ when $\cl E=   \cl E_{S\C} $. Morever, we have
   completely isometric embeddings
   $${\bb G}_{\cl E_{\max}}  \subset {\bb G}_{\cl E_1}  \subset  {\bb G}_{\cl E_{S\C}}  .$$
   These are strict inclusions since the three corresponding classes are distinct.
   Indeed,   for instance the column space $C_n$ is  in ${\cl E_1} \setminus\ovl{ {\cl {S E}_{\max} } }$
   (see \cite[Th. 10.5, p. 222]{P4}) for all $n>1$; moreover  $M_n$ and $\ell_\infty^n$ are in  ${ \cl { E}_{S\C} } \setminus\ovl{ \cl {S E}_{1} }$ for all $n>2$.
   The latter because if 
    an injective space $E$ is  in $ \ovl{ \cl {S E}_{1} }$
   its o.s. dual $E^*$ must be 1-exact, and this fails when $n>2$
 by \cite[Th. 21.5, p. 336]{P4}. Since $M_n$ and $\ell_\infty^n$ have the LP 
 (they are nuclear !) they are in ${ \cl { E}_{S\C} }$.

    \end{rem}
      
            The class $\cl E_{\max}$ is actually the smallest league containing $\CC$.
  We chose to denote it by $\cl E_{\max}$ because
the $d_{cb}$-closure of  $\cl E_{\max}$ is the class of f.d. maximal o.s.
      meaning f.d. spaces equipped with their maximal o.s. structures
      in the sense of Blecher-Paulsen in \cite{BP} (see also e.g. \cite[ \S 3]{P4}),
      that can be defined as follows:
            an o.s. $X$ is maximal if and only if $\|u\|_{cb}= \|u\|$
            for all $u: X \to B(H)$ and all $H$.
            The class of f.d. maximal o.s. reappears below  in Remarks \ref{rmax}
       and \ref{r18/3}     as the class $\cl E_{\max}^{[1]}$, among more general variations.
           
       We say that  $X$   is strongly maximal if, for any $\vp>0$,   any
           f.d. $E\subset X$ is contained in a larger f.d. $F\subset X$  
           such that $\forall H, \forall u: F \to B(H)\quad \|u\|_{cb}\le (1+\vp) \|u\|.$
            When $X$ is separable this means
       there is an increasing sequence of f.d. subspaces $X_n\subset X$
      with dense union such that  for some  $\vp_n\to 0$
       we have
       $$\forall H, \forall u: X_n \to B(H)\quad \|u\|_{cb}\le (1+\vp_n) \|u\|.$$
       
     \begin{rem}[Some important questions]\label{iq}  Clearly, strongly maximal implies maximal. The converse is 
      apparently open (even if we only  require  for $(\vp_n)$ to be a bounded sequence). We conjecture that the converse fails  
 for lack of a suitable approximation property in $X$.
 Actually it is unclear whether any maximal space locally embeds
 in $\cl E_{\max}$ (this was raised already in \cite{[O3]}). 
 Even whether it locally embeds in $\cl E_{S\C}$ is unclear to us.
 In fact we do not know if the class of f.d. subspaces of maximal o.s.
 (equivalently the class of all f.d. subspaces of $B(\ell_2)$
 equipped with its maximal o.s. structure)  is separable.
 A negative answer would be a very significant strengthening of 
 the non separability of $OS_{fd}$ proved in \cite{JP}.
 The failure of the approximation property for $B(\ell_2)$ (see \cite{Sz}) seems to play a role here
 behind the scene.
 These queries are related to some of Ozawa's questions in \cite{Ozllp} (see also \cite[\S 6]{[O3]}).
    \end{rem}

         \begin{cor}\label{dc1} There is a separable  o.s. ${\cl X}_{\max}\subset B(H)$    that is $\cl E_{\max}$-Gurarii, maximal
         and even strongly maximal, with the strong OLLP and 
         for which there is a (normal) projection $P: B(H)^{**} \to {\cl X}_{\max}^{**}$ with $\|P\|=1$.
      \end{cor}
      \begin{proof} By Proposition \ref{p12/3}
      we may apply Theorem \ref{t4/2p} to obtain our $\cl E_{\max}$-Gurarii space.
      Let ${\cl X}_{\max}={\bb G}_{\cl E}$ for $\cl E=\cl E_{\max}$.
      Since  ${\cl X}_{\max} =\ovl{\cup E_n}$ with $E_n$ f.d. and maximal strong maximality follows.
       The last assertion follows from  (ii)' in 
      Proposition \ref{p1}. 
       \end{proof}
    
 \begin{rem}
  One might be tempted to guess that the classical Gurarii space
 (which is a $\cl L_\infty$-space, see Remark \ref{bex'}) equipped with its maximal o.s. structure could have the 
 properties in Corollary \ref{dc1}, but it is not so.
Indeed, by Theorem \ref{t4/2p}   when $\cl E=\cl E_{\max}$
the space $\bb G_{\cl E}$ contains $\ell_1$  completely isometrically.
 But if a maximal o.s. $X$  
   contains $\ell_1$ completely isomorphically, then $X$ has quotients
   uniformly  isomorphic (in the Banach space sense) to $M_n$ ($n\ge 1$), 
   and hence cannot be isomorphic to a $\cl L_\infty$-space.
   Indeed, any  metric surjection $\ell_1\to M_n$ 
   is a fortiori a complete contraction, and by the injectivity of $M_n$
   it extends to a (complete) contraction from $X$ onto $M_n$,
   which is also a metric surjection.
   The fact that the spaces $\{M_n\mid n\ge 1\}$ are not uniformly isomorphic
   to quotients of $\cl L_\infty$-spaces, or equivalently that
   $\{M^*_n\mid n\ge 1\}$ does not uniformly embed in $\cl L_1$-spaces
   goes back to Gordon and Lewis \cite{GL} (see also \cite{PP}).
   The same argument shows that $\bb G_{\cl E_{\max}}$ cannot have local unconditional structure in the sense of \cite{GL}.
 \end{rem}

 \begin{rem} Friedman and Russo  \cite{FR} (see also \cite{FR2}) 
 proved that the range of a contractive projection $P$
 on a $C^*$-algebra is  a Jordan triple system
 for the triple product $(a,b,c)\mapsto P(ab^*c +cb^*a)/2$
 that 
  has a faithful representation as a $J^*$-algebra.
 A $J^*$-algebra is  closed subspace of $B(H,K)$
 stable by the mapping $x\mapsto xx^*x$.
 This applies to ${\cl X}_{\max}^{**}$.
  \end{rem}
 \begin{rem} 
 Note that, by
 Ozawa's Theorem \ref{op1}, a f.d.o.s. $E$ has the OLLP
 if and only if $E\in \ovl{\cl E_1}$.
 Since   $\ell_\infty^n\not\in \ovl{\cl E_1}$  when $n>2$ (see \cite[p. 336]{P4}) and $\ell_\infty^n$ is injective,
 an o.s. with OLLP cannot contain  $\ell_\infty^n$ completely isometrically.
 This shows that
 no infinite dimensional $C^*$-algebra has the OLLP. In particular
  the $\cl E$-Gurarii spaces  for $\cl E\subset \ovl{\cl E_1}$
 are not $C^*$-algebras.
 
 \end{rem}
        \begin{rem}\label{rmax}
        Fix an integer $N\ge 1$.
        Let $\cl E_{\max}^{[N]}$ be the class formed of all f.d. quotients
        of o.s. of the form $\ell_1 (S_1^N)$. 
        This is the class of f.d. $M_N$-maximal o.s.
        in the sense of Lehner \cite{Le}.
         See \S \ref{oik} for more details on $M_N$-spaces.\\
        Then $\cl E_{\max} \subset
        \cl E_{\max}^{[N]} \subset \ovl{\cl E_1}$.
        Moreover it is easy to check that
        $\ovl{\cl E_1}=\ovl{\cup_N  \cl E_{\max}^{[N]}}$,
        and the latter class coincides
        with that of f.d.o.s. $E$ such that $E^*$ is $1$-exact. This should be compared with Remark \ref{r18/3}.
         The space ${\bb G}_{\cl E_{\max}[N]} \subset B(H)$
        has the OLLP and is such that there is a projection 
        $P: B(H)^{**}\to {\bb G}_{\cl E_{\max}[N]}^{**}$ with $Id_{M_N} \otimes P$
        contractive. The proof is similar to that of Proposition \ref{p1w}.
         \end{rem}
 
       By Proposition \ref{p1w} whenever $\cl E\supset \cl E_1$ the space ${\bb G}_{\cl E}$
       is an ${\cl E}$-Gurarii space
      with the WEP. In particular:

      \begin{cor}\label{dc3} There is a separable o.s. ${\cl X}_1\subset B(H)$ with the strong OLLP that also has  the WEP. In particular 
         there is a (normal) projection $P: B(H)^{**} \to {\cl X}_1^{**}$ with $\|P\|_{cb}=1$.
    \end{cor}
     \begin{proof} 
     By Proposition \ref{p12/3}
      and Theorem \ref{t4/2p}, we may  just set  ${\cl X}_1=\bb G_{\cl E_1}$. 
        \end{proof}
 \begin{rem}\label{tro}  By a result due to Youngson \cite{Y} 
 the range of a completely contractive projection $P$ on a $C^*$-algebra $B$
 is (completely isometric to) a triple subsystem of another $C^*$-algebra $C$
 with  triple products given respectively by $[a,b,c]=P(ab^*c)$ ($a,b,c \in P(B)$)
 and $\{x,y,z\}= xy^*z$ ($x,y,z\in C$).
 Moreover,  for any $x,y,z \in B$ we have
 $$P(P(x)P(y)^* P(z))=P(xP(y)^* P(z))=P(P(x)y^* P(z))=P(P(x)P(y)^*z).$$
 This implies that $P(B)$ is completely isometric to a ternary ring of operator (TRO).
 By definition, a TRO is a subspace of $B(H,K)$ ($H$,$K$ Hilbert spaces)
 (or simply of a $C^*$-algebra) that is stable under the triple product $(x,y,z) \mapsto xy^*z$. Kirchberg in \cite{Kcar1} uses the term $C^*$-triple system.
 By \cite[Prop. 4.2]{Kcar1} (see also \cite{KR}) a TRO can be identified
 (as a TRO and completely isometrically) with $pA(1-p)$  where $A$ is its ``linking" $C^*$-algebra for some projection $p\in A$.
 In the weak* closed case $A$ is a von Neumann algebra.
 This applies in particular to the space ${\cl X}_1^{**}$ in Corollary \ref{dc3}.
 Perhaps the underlying linking von Neumann algebra deserves more
 investigation.
      We refer the reader to \cite{BLM} for more  results   on triple systems and TROs.
  \end{rem}

    \section{Uniqueness of $\cl E$-Gurarii spaces}\label{uniq}

  Recall that a linear map $f:E \to F$ between o.s.
     is called an $\vp$-embedding if it is injective
     and $$\max\{ \|f   \|_{cb} , \|{f^{-1} }_{|f(E)} \|_{cb} \} \le 1+\vp.$$
    
 To prove the uniqueness up to complete isometry of Theorem \ref{t5}
 we will use the same idea as in \cite{Kus}.

 \begin{dfn}\label{d23.12} We will say that a Gurarii class $\cl E$ is perturbative if 
 there is a function $\d: (0,\infty) \to (0,1) $ with $\lim\nl_{\vp\to 0} \d(\vp)=0$
 such that  for any $\vp>0$ the following holds:
 for any $L\in \cl E$, $S\subset L$,   $E\in \cl E$ 
   and any  $\vp$-isometric $u: S\to E$, 
 there is for  any $\d'>0$ a space $\hat E_1\in \cl E$,
  and $\d'$-isometric maps $\phi: E \to \hat E_1$ and  $\tilde u: L \to \hat E_1$ 
  such that
  $\|{\tilde u}_{|S} - \phi u\|_{cb}\le \d(\vp)$.
 \end{dfn}
  \begin{rem}\label{r12/3} Again it is easy to check that
  if $\cl E$ is perturbative then so are $\cl{S E}$ 
  and $\ovl{\cl{S E}}$.
  \end{rem}
 
 This   property expresses roughly that if the range is suitably enlarged (within $\cl E$)
 any  $\vp$-isometric $u: S\to E$ 
 is the restriction of a $\d(\vp)$-perturbation of    a map (namely $ \tilde u$)
 that is almost isometric. Since $ \tilde u$ is ``more isometric"
 than $u$, of course there has to be a compensation and $\d(\vp)$ has to be essentially larger than $\vp$,
 but still if $\d(\vp)\to 0$ when $\vp\to 0$, as we will see, this leads to some
 strong consequences.
 
       \begin{lem}\label{1} 
    If $\cl E$ is a league then for all $\vp>0$, 
     all $E,F\in \cl E$ and all $\vp$-isometric $f: E \to F$ 
    there is $\cl Z\in \cl E$ and completely isometric maps
    $i: E \to \cl Z$ and $j: F \to \cl Z$ such that
    $$ \|jf-i\|_{cb} \le \vp.$$
       \end{lem}
    \begin{proof} The proof is similar to the one given by Lupini
     in \cite[Lemma 3.1]{L0}.
     Consider $\cl Z=E\oplus F\oplus E$
    with norm
    $$\|(x,y,z)\|= \|x\|+\|y\|+\vp\|z\|.$$
    We will quotient this by
    $$N=\{ (-e,f(e),e ) \mid e\in E\}.$$
    Let $q: \cl Z \to \cl Z/N$ denote the quotient map. We set $Z= \cl Z/N$.
Then we set
$i(x)= q( (x,0,0) )$ ($x\in E$)
and $j(y)= q( (0,y,0) )$ ($y\in F$).
Note
$$ j(f(x))-i(x)= q( (-x,f(x), 0)  )= q( (-x,f(x), x)  )+ q( (0,0, -x)  )=q( (0,0, -x)  ) $$
and hence
$$\|jf-i\|\le \vp.$$
We have clearly $\|i\|\le 1$ and $\|j\|\le 1$.
Moreover,
for any $e\in E$
$$\|  (x,0,0)+ (-e,f(e),e )\|=\|x-e\|+\|f(e)\|+\vp\|e\|\ge\|x\|-\|e\|+\|f(e)\|+\vp\|e\|$$
$$ \ge \|x\|-\|e\| +\|e\|(1+\vp)^{-1}+\vp\|e\| \ge \|x\|,$$
and hence $\|i(x)\|\ge \|x\|$.
Similarly
$$\|  (0,y,0)+ (-e,f(e),e )\|=\|-e\|+\|y+f(e)\|+\vp\|e\|\ge\|e\|+\|y\|- \|f(e)\|+\vp\|e\| \ge \|y\|,$$
and hence $\|j(y)\|\ge \|y\|$.
This gives us isometries $i,j$ and  Lemma \ref{1} with the usual
norms instead of cb-ones. To pass to cb-norms
we use as earlier the identity
$\|f\|_{cb}= \|Id_{S_1} \otimes f: S_1[E] \to S_1[F]\|$ valid for any $f: E \to F$.
We equip $\cl Z$ with the o.s.s. of the direct sum
in the sense of $\ell_1$ (with the third factor weighted by $\vp$) so that $S_1[\cl Z]\simeq S_1[E]\oplus S_1[F]\oplus S_1[E]$
and for any $(x,y,z) \in S_1[E]\oplus S_1[F]\oplus S_1[E]$
we have
$$\|(x,y,z)\|_{S_1[\cl Z]}= \|x\|_{S_1[E]}+\|y\|_{S_1[F]}+\vp\|z\|_{S_1[E]}.$$
Moreover, 
$${S_1[N]}\simeq \{ (-e,( Id_{S_1} \otimes f)e,e )\mid e\in S_1[E]\}.$$
Thus if we run the preceding argument with $f$ replaced by $Id_{S_1} \otimes f: S_1[E] \to S_1[F]$ and
$\cl Z$ replaced by $S_1[\cl Z]$ we obtain the announced statement with cb-norms and with $\d(\vp)=\vp$.
    \end{proof}
      \begin{lem}\label{23} 
    If $\cl E$ is a league then it is a perturbative Gurarii class.    \end{lem}
    \begin{proof}
    Let $S\subset L$ and  $u:S \to E$ be as in Definition \ref{d23.12}.
    We know by Proposition \ref{p12/3}
    and 
    Remark \ref{r11/3}
    there is $E_1\in \cl E$, a complete isometry $k: E \to E_1$
    and an $\vp$-isometry $f: L \to E_1$
    such that ${f}_{|S} = ku$. Applying the preceding lemma
     we find $\cl Z\in \cl E$ with  $i: L \to \cl Z$
    and $j: E_1\to \cl Z$ such that 
    $\|j f - i\|_{cb}\le \vp$. Let $\hat E_1= \cl Z$, $\phi=jk$  and  $\tilde u=i$, so that
    $\phi u= jf_{|S}$.
    We obtain the property in Definition \ref{d23.12} with $\d'=0$ and $\d(\vp)=\vp$.
    \end{proof}
    There is a quite different sort of perturbative Gurarii class that we
    describe next.      
    \begin{lem}\label{18/3} Let $\cl E$ be a class of f.d.o.s. Assume that 
      \item{\rm (i)}   $\cl E$ is stable by $\oplus_{\infty}$,
      meaning that for any   $L,E\in \cl E$
       we have  $L \oplus_\infty E \in \cl E$.
       \item{\rm (ii)} 
      Each $E\in \cl E$ is injective (meaning that there is a
    completely isometric embedding $E\subset B(H)$ and a projection $P: B(H) \to E$ with $\|P\|_{cb}=1$).\\
    Then $\cl E$ is a perturbative tightly Gurarii class.
    \end{lem}
    \begin{proof}  Let us first show that $\cl E$ is tightly Gurarii.
 Consider $L,E\in \cl E$ and $u:S \to E$ injective as in Definiton \ref{d5/2}.
 By Lemma \ref{l1} there is a f.d.o.s. $E_1 $, $\tilde u: L \to E_1$
 and a complete isometry $\phi: E \to E_1$ such that  ${\tilde u}_{|S}= \phi u$
 and $D({\tilde u}) = D(u)$. 
 We claim that there is a space $\hat E_1\in \cl E$ and 
 an injective map $\psi : E_1 \to \hat E_1$ 
 with $\| \psi\|_{cb} =1$
 such that the maps $\hat u: L \to \hat E_1$ and $\hat \phi: E \to \hat E_1$
 defined by $\hat u=\psi \tilde u$ and $\hat \phi=\psi\phi$ 
 (and hence ${\hat u}_{|S}= \hat\phi u$) satisfy the requirements to show that $\cl E$ is tightly Gurarii.
 
 The space $\hat E_1$ is defined simply as  $\hat E_1=L\oplus_\infty E$.
 The map $\psi : E_1 \to \hat E_1$ is defined by setting
 $\psi (x)= \|w\|_{cb}^{-1}w(x)\oplus \|v\|_{cb}^{-1} v(x)$ where $w$ and $v$ are as follows.
 Consider ${\tilde u}^{-1} : {\tilde u}(L) \to L $. Since $L$ is injective
 the latter map admits an extension 
  $w: E_1 \to L$  with $\|w\|_{cb}= \| { {\tilde u}^{-1} }_{|{\tilde u}(L)}  \|_{cb}$.
  Similarly, since $E$ is injective  the map ${\phi}^{-1} : {\phi}(E) \to E $  admits an extension 
  $v: E_1 \to E$  with $\|v\|_{cb}= 1$.
  Note that $\hat u(y)= \psi \tilde u(y)= (\|w\|_{cb}^{-1}y, \ast)$ ($y\in L$)
and $\hat \phi(e)=\psi \phi (e)=( \ast , e)$ ($e\in E$). So we have
  $$\|  {\hat u^{-1}}_{|\hat u(L) } : \hat u(L) \to L\|_{cb}  \le \|w\|_{cb} = \| { {\tilde u}^{-1} }_{|{\tilde u}(L)}  \|_{cb},$$ 
  and hence  
 $$ D(\hat u )  \le  D(\tilde u )  = D(u).$$
Moreover $\hat \phi$ is completely isometric.
  By our assumptions on $\cl E$ we know that $\hat E_1\in \cl E$.
  Thus we conclude that $\cl E$ is tightly Gurarii.
  The same argument can be repeated to show that 
  $\cl E$ is perturbative.
    \end{proof}

  \begin{lem}\label{L2}
  Assume that $\cl E$ is perturbative with associated function $\vp\mapsto \d(\vp)$.
   Let $f: E \to F$
    be a complete $\vp$-embedding ($\vp>0$) with $E,F\in \cl E$.
    Assume that $E$ is a f.d. subspace of an $\cl E$-Gurarii space $X$.
    Then for any $0<\vp'<1$ there is
    a subspace $E_1\subset X$ containing $E$ and a complete $\vp'$-embedding
    $g: F \to E_1$ such that  
       $$\|(gf-Id_{E_1})_{|E}\|_{cb} \le 2 (1+\vp)\d( \vp)  .$$    \end{lem}
    
\begin{proof} 
Let $i_E:E \to X$ denote the inclusion.
We apply  the perturbative  Gurarii property with $ L=F$, $S=f(E)\subset L$
and $u={f^{-1}}_{|f(E)}$.
 For any $\vp''>0$ this gives us $\hat E_1\in \cl E$, and $\vp''$-isometric maps
$\phi: E \to \hat E_1$ and  $\tilde u: F \to \hat E_1$ such that
$\|{\tilde u}_{|S} - \phi u\|\le \d(\vp)$.
By the $\cl E$-Gurarii property of $X$
  applied to $\phi^{-1}: \phi(E) \to E$, there is
  a map $w: \hat E_1 \to X$
  such that   $w_{|\phi(E)} =i_E \phi^{-1}$ or equivalently $w\phi=i_E    $   and (recall \eqref{21.12'})
   \begin{equation}\label{e10/3} 
   D'(w) \le
    (1+\vp'') D'(\phi)  \le  (1+\vp'')^2.\end{equation}
  We then set $g= w\tilde u: F \to X$. 
  We have  $
  D'(g)  \le (1+\vp'')^3$ and 
  $$gf-i_E = w {\tilde u}_{|S} f  -  i_E f^{-1} f  = (w {\tilde u}_{|S}    -  i_E u) f  
  = (w {\tilde u}_{|S}    - w\phi u) f 
   $$ and hence  by \eqref{e10/3}
  $$\|(gf-i_E) \|_{cb} \le \|w \|_{cb}  \|{\tilde u}_{|S} - \phi u\|_{cb}  \|f \|_{cb}     \le   (1+\vp'')^2 \d(\vp) (1+\vp).$$
  Let  $E_1\subset X$ be any f.d. subspace
  such that $g(F) + E \subset E_1$, we may then view $g$ and $i_E$ as 
  having range in $E_1$ and it only
 remains to choose $\vp''$ small enough with respect to $\vp'$ to obtain the announced result.
\end{proof}

Mimicking the approach of \cite{Kus} we will prove

    \begin{thm}\label{T1} 
    Assuming that $\cl E$ is a perturbative   Gurarii class with associated function $\vp\mapsto \d(\vp)$. Let $X,Y$ be $\cl E$-Gurarii spaces, let $\vp>0$, 
      and let $E\subset X$ be a f.d. subspace in $\cl E$.
    Let $u : E \to Y$ be a complete $\vp$-embedding.   
    Then for any $\d'>0$ there
    is a completely isometric isomorphism  $f: X \to Y$
    such that $\|f_{|E} -u\|_{cb} \le 8\d(\vp) +\d'$.\\
    In particular, $X$ and $Y$ are completely isometrically isomorphic.
      \end{thm}

      \begin{proof} 
     The proof can be completed exactly as in \cite[p. 453]{Kus}
     but using our Lemma \ref{L2} in place of
     \cite[Lemma 2.2]{Kus}. Unfortunately, our notation
     clashes with that of \cite{Kus} so we   repeat the
     argument of \cite{Kus} for the reader's convenience.
     
     Fix $\vp_0=\vp$. Let $(\vp_n)_{n>0}$ be a  sequence of   numbers such that 
     $0<\vp_n<1$ for all $n>0$  and $\sum \d(\vp_n)<\infty$ (to be specified).
     By induction, following \cite{Kus} we construct sequences $(E_n)$ and $(F_n)$
     of f.d. subspaces respectively of $X$ and $Y$,
     together with maps $(f_n)$ and $(g_n)$
     satisfying the following conditions:
     \item{(0)} $E_0=E$, $F_0=u(E)$,
        \item{(1)} $f_n : E_n \to F_n$ is an $\vp_n$-embedding,  
           \item{(2)} $g_n : F_n \to E_{n+1}$ is an $\vp_{n+1}$-embedding,
              \item{(3)} $\|(g_nf_n -Id_{E_{n+1}})_{|E_n}\|_{cb}\le 4 \d(\vp_n)  \|x\|$ and $E_n\subset E_{n+1}$,
                 \item{(4)} $\|(f_{n+1} g_n -Id_{F_{n+1}})_{|F_n}\|_{cb}\le 4\d(\vp_{n+1}) \|y\|$ and $F_n\subset F_{n+1}$,
                    \item{(5)}  
                    $X=\ovl{\cup E_n}$, $Y=\ovl{\cup F_n}$.
                    
                    The proof can be illustrated by the following asymptotically approximately commuting
                    diagram.
                     $$\xymatrix{  X \ar[r]^{f}& Y\ar[r]^{g} & X
   \\E_{n+1}\ar@{-->}[r]^{f_{n+1}}\ar@{^{(}->}[u]  & F_{n+1} \ar@{^{(}->}[u]  \ar@{-->}[r]^{g_{n+1}} & E_{n+2}  \ar@{^{(}->}[u] \\
 E_n\ar@{-->}[r]^{f_n}\ar@{^{(}->}[u]& F_n \ar@{^{(}->}[u]\ar@{-->}[r]^{g_n} & E_{n+1}\ar@{^{(}->}[u] }$$ 
                    
                    We start with $E_0=E$, $F_0=u(E)$ and $f_0=u$.
                \\
                    Assume we know $  E_k,F_k,f_k$ for all $k\le n$.
                    Then by Lemma \ref{L2} applied to $f_n$
                    there is  $E_{n+1}$,  that we can pick so that $E_{n+1} \supset E_n$,  and an 
                    $\vp_{n+1}$-embedding
                    $g_{n}:  F_n \to E_{n+1}$ such that (3) holds.
                    With the latter map $g_n: F_n \to E_{n+1}$ at hand,  we apply Lemma \ref{L2} to  it                    (now with $Y$ in place of $X$), this gives us $F_{n+1}$,  that we can pick so that $F_{n+1} \supset F_n$, and 
                    an $\vp_{n+1}$-embedding
                    $f_{n+1} : E_{n+1} \to F_{n+1}$ such that (4) holds.
                    \\
                    In particular,  starting from $E_0,F_0,f_0$ we obtain $g_0$   as well as $E_1,F_1,f_1$,
                    and then $g_1$   as well as $E_2,F_2,f_2$ and so on.
                    The condition (5) can easily be ensured by enlarging if necessary
                    the space $E_n$ (resp. $F_n$) at the $n$-th step so that   it contains the first $n$ elements
                    of a dense sequence in $X$ (resp. $Y$).   
                 Then (3) implies
                    $\|(f_{n+1}g_nf_n -f_{n+1})_{|E_n}\|_{cb}\le 4\d(\vp_n) \|f_{n+1}\|_{cb} \le 8\d(\vp_n)  $  while (4)
                    implies
                    $\|f_{n+1}g_nf_n -f_{n}\|_{cb}\le 4\d(\vp_{n+1}) \|f_n\|_{cb}\le 8\d(\vp_{n+1}) $. Therefore by the triangle inequality
                     \begin{equation}\label{err1} 
                \|(f_{n+1} -f_{n})_{|E_n}\| \le       \|(f_{n+1} -f_{n})_{|E_n}\|_{cb}
                \le 8\d(\vp_n)   +8\d(\vp_{n+1})  .\end{equation}
                    Since $(\d(\vp_n))$ is summable, $(f_n(x))$ converges in $Y$
                    to a limit $f(x)$ for any $x\in \cup E_n$. Clearly $f$ extends to a complete isometry
                    from $X $ to $Y$.
                    Arguing similarly with $g_{n+1} f_{n+1} g_n$ we find that
                    $(g_n(y))$ converges in $X$
                    to a limit $g(y)$ for any $y\in \cup F_n$ that defines a complete isometry
                    $g: Y\to X$. Clearly (3), (4) and (5) imply that
                    $fg=Id_Y$ and $gf=Id_X$, whence the conclusion
                    that $f$ is a bijective complete isometry.                    
                      Since $f_0=u$ and $E=E_0\subset E_n$ the bound \eqref{err1} gives us
                      $$\|f_{|E} -u\|_{cb} \le \sum\nl_0^\infty \|(f_{n+1} -f_{n})_{|E_n}\|_{cb}
                    \le \sum\nl_0^\infty8\d(\vp_n)   +8\d(\vp_{n+1}) =
                    8\d(\vp)+16\sum\nl_1^\infty \d(\vp_n),$$
                    so that by a suitable choice of  $(\vp_n)_{n>0}$ we can ensure that
                    $\|f_{|E} -u\|_{cb} \le 8\d(\vp) +\d'$.
   \end{proof}
    
      Applying this for an arbitrary choice of $E$,  we state for emphasis:
      
    \begin{cor}[Uniqueness]\label{cd1}  If $\cl E$ is  a perturbative  Gurarii class, any two $\cl E$-Gurarii spaces are
    completely isometrically isomorphic.
    \end{cor}

\begin{rem}[More examples]\label{mex}
Let $\cl E$ be a family of injective f.d.o.s.
stable by
$\oplus_\infty$-direct sums.  
By Lemma \ref{18/3}, $\cl E$ is perturbative Gurarii and hence there 
is a unique space $\bb G_\cl E$.\\
Let $Z$ be   a fixed injective f.d.o.s. 
 Let  $\cl E(Z)$ be   the class formed 
of all the $\oplus_\infty$-direct sums of finitely many copies of $Z$.
We can associate to $Z$ the space $\bb G_{\cl E(Z)}$.
Perhaps the cases of $Z=R_n$ (row $n\times n$-matrices) or $Z=C_n$ (column  $n\times n$-matrices) 
 or more generally $Z=M_{p,q}$ (rectangular matrices of size $p\times q$) deserve further investigations.\\
 We  discuss the particular case $Z=M_N$ in the next section.
\end{rem}

    \section{Oikhberg's exact Gurarii space}\label{oik}
    
    In \cite{Oi} Oikhberg proved the existence of an analogue of the Gurarii space among exact operator spaces. While the spaces of interest to us
    in this note are mainly non exact, we would like to indicate how the existence and uniqueness
    of the exact Gurarii space can be derived from Theorems \ref{t4/2p}
    and   \ref{T1}.
      Note that, using \cite{Kus},  Lupini proved in \cite{L0} the uniqueness
    of Oikhberg's space up to completely isometric isomorphism. Moreover, in \cite{L0,L1} (see also \cite{GoL}) Lupini placed the whole
    subject of Gurarii operator spaces in the much broader context of Fra\"\i ss\'e limits, which were connected to Gurarii spaces    by Ben Yaacov (see \cite[\S 3.3]{BY})  inspired by Henson's unpublished  work.  The more recent paper   \cite{FLAMT} studies Fra\"\i ss\'e limits for the class
    of  Banach spaces that are $L_p$-spaces or lattices.
    In this framework (with which we confess we are not too familiar)
    our paper probably just boils down to providing more examples illustrating the applicability of Fra\"\i ss\'e limits in operator space theory,
    beyond Oikhberg's exact space.

  Let $\cl E^{[N]}_{\min}$ be the collection of all f.d. operator 
subspaces of $\ell_\infty( M_N)$ (or equivalently of $\ell_\infty\otimes_{\min} M_N$), and let
$\cl E^{[\infty]}_{\min}=\ovl{\cup_N \cl E^{[N]}_{\min}}$.
The set $\cl E^{[\infty]}_{\min}$ coincides with the collection of all 
f.d. $1$-exact o.s.

\begin{lem}\label{21/3}
The classes $\cl E^{[N]}_{\min}$ and $\cl E^{[\infty]}_{\min}$ are perturbative Gurarii.
\end{lem}
\begin{proof} Let $ \cl E[M_N]$ be as in Remark \ref{mex}
 the perturbative Gurarii class  formed 
of all the $\oplus_\infty$-direct sums of finitely many copies of $M_N$.
It is easy to check that $\cl E^{[N]}_{\min}=\ovl{ \cl{S  E}[M_N]  }$.
By Remark \ref{r12/3} the class $\cl E^{[N]}_{\min}$ is   perturbative Gurarii.
Since $ (\cl E^{[N]}_{\min})$ is  monotone increasing its union
   ${\cup_N \cl E^{[N]}_{\min}}$ is still perturbative Gurarii, and 
   by Remark \ref{r12/3} so is its closure $\cl E^{[\infty]}_{\min}$.
\end{proof}

 Let $\G_N$  and $\G_\infty$ be respectively the  unique $\cl E^{[N]}_{\min}$-Gurarii space and the unique $\cl E^{[\infty]}_{\min}$-Gurarii space
 (see Theorems \ref{t4/2p}
    and   \ref{T1}).

Note that when $N=1$ the space $\G_1$ is just the classical
     Gurarii space equipped with its minimal o.s. structure (see Remark \ref{bex'}). 
 The space $\G_\infty$ is 
  Oikhberg's exact Gurarii space from \cite{Oi},
  i.e. the unique $\cl E$-Gurarii space when $\cl E$ is the class of $1$-exact f.d.o.s.

    \begin{rem}\label{r18/3}    The class $\cl E^{[N]}_{\min}$ 
    (resp. $\cl E^{[N]}_{\max}$) is the class of all f.d. $M_N$-minimal
    (resp. $M_N$-maximal)  o.s. in the sense of Lehner
    \cite{Le}. 
When $N=1$   these are just the f.d. minimal (resp. maximal) o.s. 
Note that for simplicity we previously denoted $\cl E^{[1]}_{\max}$ by $\cl E_{\max}$.

 The notion of $M_N$-space in the sense of Lehner
    \cite{Le} provides a very clear picture of the properties
    of $\bb G_N$, as an intermediate object between
    the classical Gurarii space and Oikhberg's exact variant.
    To describe this we first give  a quick review of   $M_N$-space theory.

 By an $M_N$-space we mean 
    (following \cite{Le}) a vector space $V$ equipped with
    a norm $\|\cdot \|_{(N)}$ for which there is
    an embedding $V \subset B(H)$ with which 
    $\|\cdot \|_{(N)}$ coincides with the norm
    induced by $M_N(B(H))$. The morphisms $u: V_1\to V_2$
    between $M_N$-spaces are now the maps such that
    $Id_{M_N} \otimes u: M_N(V_1) \to M_N(V_2)$ is bounded
    and the cb-norm is now replaced by
    $$\|u\|_N:=\| Id_{M_N} \otimes u: M_N(V_1) \to M_N(V_2)\|.$$
    A map with $\|u\|_N\le 1$ 
    (resp. $\|u\|_N\le 1$ and $\|u^{-1}\|_N\le 1$ )
    is called $N$-contractive (resp. $N$-isometric).
    
    Just like operator spaces, $M_N$-spaces admit a duality,
     as well as natural notions of quotient space,  of direct sums
     $\oplus_\infty$ and $\oplus_1$ and more.\\
    When $N=1$, we recover the usual Banach spaces,
    the latter notions become the usual ones and
     $\|u\|_N=\|u\|$. \\When $N=\infty$ (where we replace $M_N$ by the algebra of compact operators on $\ell_2$) we recover operator space theory.

    Given an $M_N$-space $V$, there are operator space structures
    $MIN_N$ and $MAX_N$ on $V$, with associated o.s. denoted by 
    $MIN_N(V)$ and $MAX_N(V)$
    such that the maps
    $MAX_N(V) \to V \to MIN_N(V)$ are $N$-isometric
    and these are extremal in the sense that
    for any o.s. $E,F$, 
    for any $u: V \to E$ and any $v: F \to V$ we have
    $$  \|u\|_N= \| u:  MAX_N(V) \to E\|_N=  \| u:  MAX_N(V) \to E\|_{cb},$$
    $$  \|v\|_N= \| v:  F  \to MIN_N(V)  \|_N=  \| v:  F  \to MIN_N(V) \|_{cb}.$$
    \end{rem}

    Given an o.s. $E\subset B(H)$, 
    the norm on $M_N(E)$ determines the structure of an $M_N$-space
    on $E$; let us denote by $E[N]$ the latter $M_N$-space. Note
    that $M_n(E[N])=M_n(E)$ for all $n\le N$ and 
    when passing from $E$ to $E[N]$  we roughly ``forget" these norms for $n>N$.
  We  then  define a new o.s. $E^{[N]}$ by setting
   $$E^{[N]}= MIN_{N} (E[N]) .$$ 
    Then $M_n(E^{[N]})=M_n(E)$ for all $n\le N$, 
    and for any o.s. $F$ and any $u:F \to E$ we have
    $\|u:F \to E\|_{N}= \|u:F \to E^{[N]}\|_{cb}$. In particular
    the identity map satisfies $\| Id: E  \to E^{[N]}\|_{cb} =1$.
    This means  $\| \cdot\|_{M_n(E^{[N]})} \le \| \cdot\|_{M_n(E )}$
    for all  $n>N$ and  the norms on $M_n(E^{[N]})$ for $n>N$ are  the largest possible ones (as a sequence satisfying the o.s. axioms)
    satisfying the preceding two properties.
    One way to realize an embeding $E^{[N]} \subset B(\cl H)$ is to consider
      the embedding  $\Phi_N$
    defined below. Let $I_N=\{u: E \to M_N \mid \|u\|_{cb}\le 1\}$. 
    Note that by a well known
    lemma (due to Roger Smith, see e.g. \cite[p. 26]{P4}) $\|u\|_N=\|u\|_{cb}$ for all $u\in I_N$.
    Let $\Phi_N : E \to \ell_\infty( I_N; M_N) $ be the mapping defined by
    $$\forall x\in E\quad \Phi_N(x)=( u(x))_{u\in I_N}.$$
    Then $\Phi_N$ defines a completely isometric embedding of
    $E^{[N]}$  in $\ell_\infty( I_N; M_N) $ (and a fortiori in some $B(\cl H)$).
   
    This shows that
     $$E\in \cl E^{[N]}_{\min}\Leftrightarrow E= E^{[N]}.$$ 
     Note that if $F \in \cl E^{[N]}_{\min}$,  any invertible $u:F \to E$
     satisfies $$\|u^{-1}_{|u(F)}: u(F) \to F\|_{cb}= \|u^{-1}_{|u(F)}: u(F) \to F\|_{N}.$$    
    With these facts the following is easy to check.
    \begin{lem}\label{l12} Let $S\subset L$, $u: S \to E$ and $E_1$
    be as in Lemma \ref{l1}. Fix $N\ge 1$. 
    If the spaces $L$ and $E$ both embed (completely isometrically) in $\ell_\infty( M_N)$, then the space   $E_1^{[N]}$
    satisfies all the properties of $E_1$
    listed in Lemma \ref{l1}.
  \end{lem} 
  \begin{pro} We have  
  $\G_N\subset \G_{N+1}$ (completely isometrically) for all $N\ge 1$
 and  $\G_\infty=\ovl{\cup \G_N}$.
  \end{pro}
 
   \begin{proof}  
 Using Lemma \ref{l12} one can easily
  show that    Theorem \ref{t4/2p} is valid for    $\cl E= \cl E^{[N]}_{\min}$.
  Although the latter 
  is not stable
  by the o.s. version of $\oplus_1$,  it is stable
  by the $M_N$-space analogue of $\oplus_1$,
  which can be defined
  as in \eqref{de9} but with the supremum
  over all $u$'s with $\|u\|_N\le 1$.
   This explains transparently why the case
  of $\cl E=\cl E^{[N]}_{\min}$ is entirely analogous to those
  included in Theorem \ref{t4/2p}.
  Let $\G_N$ be the  $\cl E^{[N]}_{\min}$-Gurarii space.
  Since $\cl E^{[N]}\subset \cl E^{[N+1]}$, applying
  the $\cl E^{[N+1]}$-variant of  Theorem \ref{t4/2p}
  with $Y=\G_N$ and $X=\G_{N+1}$,
 we have $\G_N\subset \G_{N+1}$ (completely isometrically)
 and since
 $\cl E^{[\infty]}_{\min}=\ovl{\cup_N \cl E^{[N]}_{\min}}$
  it is easy to check that $ \ovl{\cup \G_N}$
 is an $\cl E^{[\infty]}_{\min}$-Gurarii space, and hence by uniqueness
 $\G_\infty=\ovl{\cup \G_N}$.
        \end{proof}

     The next statement was already stated in \cite[Remark 4.9]{L0} as a consequence
    of \cite[Lemma 3.17]{BY}.
    
 \begin{lem}  
  There are integers $(k(n))$ and   an increasing sequence $E_1\subset E_2\subset  \cdots$
  of f.d. subspaces of $\G_\infty$ with dense union
  such that $E_n\simeq M_{k(n)}$  completely isometrically.
   \end{lem}
    \begin{proof}
  Indeed, since $\cl E=\{M_n\mid n\ge 1\}$ is Gurarii, there is an
  $\cl E$-Gurarii space $X$. 
  By Theorem \ref{t4/2p} and Remark \ref{bex} (ii) we may assume that 
  there are integers $(k(n))$ and   an increasing sequence $E_1\subset E_2\subset  \cdots$
  of f.d. subspaces of $X$ with dense union
  such that $E_n\simeq M_{k(n)}$  completely isometrically.
  By Remark \ref{dr2}  the space $X$
  is automatically an $\ovl{\cl {S E}}$-Gurarii space.
  But since $\ovl{\cl{S E}}=\cl E^{[\infty]}_{\min}$ (see Remark \ref{dr2'}), we must have $X= \G_\infty$
  by the uniqueness of $\G_\infty$.
  \end{proof}

     While   the   presentation in this section is arranged to illustrate our framework,
      in essence  the proofs do not differ from  Oikhberg's ones
     from \cite{Oi}, updated with the contributions
     in \cite{Kus} and \cite{L1}.

     \medskip
    
    \n\textbf{Acknowledgement.} I am grateful to Jean Roydor
    for stimulating conversations.

  \end{document}